\documentclass[12pt,a4paper]{amsart}
\usepackage{amscd, amssymb,color,xypic}
\usepackage{hyperref, amscd}

\newcommand{\field}[1]{\mathbb{#1}}
\newcommand{\CC}{\field{C}}
\newcommand{\NN}{\field{N}}
\newcommand{\RR}{\field{R}}
\newcommand{\TT}{\field{T}}
\newcommand{\ZZ}{\field{Z}}

\newcommand{\under}{\backslash}

\newcommand{\Bb}{\mathcal{B}}

\newcommand{\Kk}{\mathcal{K}}
\newcommand{\Ll}{\mathcal{L}}

\newcommand{\Uu}{\mathcal{U}}
\newcommand{\Vv}{\mathcal{V}}
\newcommand{\lsp}{\operatorname{span}}
\newcommand{\clsp}{\overline{\lsp}}
\newcommand{\id}{\operatorname{id}}
\newcommand{\rk}{\operatorname{rank}}
\newcommand{\supp}{\operatorname{supp}}
\newcommand{\red}{\operatorname{r}}

\newcommand{\tr}{\operatorname{Tr}}
\newcommand{\Ind}{\operatorname{Ind}}
\newcommand{\Br}{\operatorname{Br}}
\newcommand{\Ext}{\operatorname{Ext}}
\newcommand{\Aut}{\operatorname{Aut}}
\newcommand{\dom}{\operatorname{dom}}

\newcommand{\Hh}{{H}}

\newcommand{\Sheaf}[1]{\mathsf{Sh}(#1)}

\newcommand{\Tgerms}{\mathcal{S}}

\newcommand{\tgcsa}[2]{\ensuremath{C^*_{\red}(#1 ; #2)}}

\newcommand{\unit}[1]{#1}
\newcommand{\fibredprod}[2]{\mathbin{\mbox{${{}_{#1}}{*}{{}_{#2}}$}}}

\begin{document}

\newtheorem{thm}{Theorem}[section]
\newtheorem{cor}[thm]{Corollary}
\newtheorem{posscor}[thm]{Possible Corollary}
\newtheorem{lemma}[thm]{Lemma}
\newtheorem{prop}[thm]{Proposition}
\newtheorem{goal}[thm]{Goal}
\newtheorem{conj}[thm]{Conjecture}
\newtheorem{notetoselves}[thm]{Note to selves}
\newtheorem{thm1}{Theorem}
\theoremstyle{definition}
\newtheorem{defn}[thm]{Definition}
\newtheorem{remark}[thm]{Remark}
\newtheorem{example}[thm]{Example}
\newtheorem{remarks}[thm]{Remarks}
\newtheorem{claim}[thm]{Claim}
\newtheorem{fact}[thm]{Fact}
\newtheorem{wish}[thm]{Wish}
\newtheorem{notation+h}[thm]{Notation}
\def\range{\operatorname{range}}

\numberwithin{equation}{section}

\title[A Dixmier-Douady theorem for Fell algebras]
{A Dixmier-Douady theorem for Fell algebras}

\author[an Huef]{Astrid an Huef}
\address{Department of Mathematics and Statistics\\
University of Otago\\
Dunedin 9054\\
New Zealand} \email{astrid@maths.otago.ac.nz}

\author[Kumjian]{Alex Kumjian}
\address{Department of Mathematics\\
University of Nevada \\
Reno, NV 89557\\
USA} \email{alex@unr.edu}

\author[Sims]{Aidan Sims}
\address{School of Mathematics and Applied Statistics\\
University of Wollongong \\
NSW 2522\\
Australia} \email{asims@uow.edu.au}



\thanks{We thank Rob Archbold, Bruce Blackadar, Iain Raeburn and Dana Williams
for helpful discussions and comments. This research was supported by the Australian Research
Council.}

\date{July 2010; revised November 2011}

\begin{abstract}
We generalise the Dixmier-Douady classification of
continuous-trace $C^*$-algebras to Fell algebras. To do so, we
show that $C^*$-diagonals in Fell algebras are precisely
abelian subalgebras with the extension property, and use this
to prove that every Fell algebra is Morita equivalent to one
containing a diagonal subalgebra. We then use the machinery of
twisted groupoid $C^*$-algebras and equivariant sheaf
cohomology to define an analogue of the Dixmier-Douady
invariant for Fell algebras $A$, and to prove our
classification theorem.
\end{abstract}

\maketitle

\section{Introduction}

The Dixmier-Douady theorem classifies continuous-trace
$C^*$-algebras with spectrum $T$ up to Morita equivalence by
classes in a third cohomology group \cite{DD}, and the
Phillips-Raeburn theorem classifies their $C_0(T)$-automorphisms
using classes in the corresponding second cohomology group
\cite{PR}. The Dixmier-Douady Theorem has been very influential
in the study of $C^*$-dynamical systems (see for example
\cite{RW:dd}), and has been applied in differential geometry
\cite{Bry}, in mathematical physics \cite{BHM, CM, MR}, and in
the definition of twisted $K$-theory \cite{Ros}. The object of
this paper is to extend the Dixmier-Douady theorem to  Fell
algebras.

A Fell algebra is a $C^*$-algebra $A$ such that every
irreducible representation $\pi_0$ of $A$ satisfies Fell's
condition: there is a positive $b \in A$ and a neighbourhood
$U$ of $[\pi_0]$ in $\hat A$ such that $\pi(b)$ is a rank-one
projection whenever $[\pi] \in U$. The spectrum of a Fell
algebra is always locally Hausdorff \cite[Corollary~3.4]{AS},
and is Hausdorff if and only if the Fell algebra is  a
continuous-trace $C^*$-algebra. The class of  Fell algebras
coincides with the class of Type~I$_0$ algebras defined by
Pedersen in \cite[\S 6.1]{Ped} as the $C^*$-algebras generated
by their abelian elements (see page~\pageref{pg-fell}).
Fell algebras are the natural building blocks for Type~I
$C^*$-algebras: every Type~I $C^*$-algebra has a canonical
composition series consisting of Fell algebras
\cite[Theorem~6.2.6]{Ped} (by contrast there always exists a
composition series consisting of continuous-trace
$C^*$-algebras, but no canonical one).

Let $T$ be a locally compact, Hausdorff space. Given a
continuous-trace $C^*$-algebra $A$ with spectrum identified with
$T$, the Dixmier-Douady invariant $\delta(A)$ belongs to a
second sheaf-cohomology group $H^2(T,\Tgerms)$. The
Dixmier-Douady classification of continuous-trace $C^*$-algebras
says that if $A$ and $B$ are continuous-trace $C^*$-algebras
with spectra identified with $T$, then $\delta(A)=\delta(B)$ if
and only if there is an $A$--$B$-imprimitivity bimodule whose
Rieffel homeomorphism respects the identifications of
$\widehat{A}$ and $\widehat{B}$ with $T$.

If we replace continuous-trace $C^*$-algebras with Fell
algebras, we must deal with locally compact, locally Hausdorff
spaces $X$. There is no difficulty with sheaf cohomology for
such spaces, but the definition of our analogue of the
Dixmier-Douady invariant $\delta(A) \in H^2(\widehat{A},
\Tgerms)$ is more involved. We tackle the problem using the
machinery of $C^*$-diagonals and of twisted groupoid
$C^*$-algebras.

A $C^*$-diagonal consists of a $C^*$-algebra $A$ and a maximal
abelian subalgebra $B$ of $A$ with properties modeled on those
of the subalgebra of diagonal matrices in $M_n(\CC)$ (see
Definition~\ref{dfn:diagonal}). Diagonals relate to Fell
algebras as follows. Consider a Fell algebra $A$ with a
generating sequence $a_i$ of pairwise orthogonal abelian
elements such that $a := \sum_i \frac{1}{i}a_i$ is strictly
positive in $A$. That is, the hereditary subalgebra generated by
$a$ is equal to $A$. Then $B := \bigoplus_i \overline{ a_i A
a_i}$ is an abelian subalgebra of $A$, which we prove is a
diagonal. Indeed, Theorem~\ref{prp:construct diagonal} shows
that every separable Fell algebra $A$ is Morita equivalent to a
$C^*$-algebra $C$ with a diagonal subalgebra $D$ arising in just
this fashion. In outline the construction is straightforward.
Fix a sequence $a_i$ of abelian elements which generate $A$ and
let $\tilde{a}_i = a_i \otimes \Theta_{i\,i} \in A \otimes
\Kk(l^2(\NN))$ for each $i$. Let $C$ be the smallest hereditary
$C^*$-subalgebra containing all the $\tilde{a}_i$ and let
\[\textstyle
D := \bigoplus_i \overline{ \tilde{a}_i (A \otimes \Kk) \tilde{a}_i}
=  \bigoplus_i \big(\overline{ a_i A a_i} \otimes \Theta_{i\,i}\big).
\]
To prove that $D$ is a diagonal, we show in
Theorem~\ref{thm:ext prop => diag}  that diagonals in Fell
algebras $A$ can be characterised as the abelian subalgebras
$B$ which have the extension property relative to $A$: every
pure state of $B$ extends uniquely to a state of $A$. This
extends \cite[Theorem~2.2]{Kumjian1986} from continuous-trace
$C^*$-algebras to Fell algebras.
Example~\ref{ex:alex} shows that this characterisation does not
generalise to bounded-trace $C^*$-algebras.

$C^*$-diagonals arise naturally from topological twists: exact
sequences of groupoids
\[
\Gamma^{(0)} \to \Gamma^{(0)} \times \TT \to \Gamma \to R
\]
(just $\Gamma \to R$ for short) such that $\Gamma$ is a
$\TT$-groupoid and $R$ is a principal \'etale groupoid with
unit space $\Gamma^{(0)}$ (see page~\pageref{defn twist}).  The
associated twisted groupoid $C^*$-algebra $\tgcsa{\Gamma}{R}$
is a completion of the space of continuous $\TT$-equivariant
functions on $\Gamma$ and contains a subalgebra isomorphic to
$C_0(\Gamma^{(0)})$. Moreover, the pair
$\big(\tgcsa{\Gamma}{R}, C_0(\Gamma^{(0)})\big)$ is a
$C^*$-diagonal. Kumjian showed in \cite{Kumjian1986} that every
diagonal pair arises in this way: given a diagonal pair $(A,B)$
there exists a topological twist $\Gamma\to R$ and an
isomorphism $\phi : A \to \tgcsa{\Gamma}{R}$ such that $\phi(B)
= C_0(\Gamma^{(0)})$. Together with the results outlined in the
preceding paragraph, this implies that each Fell algebra is
Morita equivalent to a twisted groupoid $C^*$-algebra
$\tgcsa{\Gamma}{R}$.

Given a principal \'etale groupoid $R$, an isomorphism of
twists over $R$ is an isomorphism of exact sequences which
identifies ends. The isomorphism classes of topological twists
over $R$ form a group $\mathrm{Tw}(R)$ called the twist group
\cite{Kumjian1985}. It was shown in \cite{Kumjian1988} how the
twist group fits into a long exact sequence of
equivariant-sheaf cohomology. In particular, the boundary map
$\partial^1$ in this long exact sequence determines a
homomorphism from the twist group to the second
equivariant-cohomology group $H^2(R, \Tgerms)$. We use this
construction to define an analogue of the Dixmier-Douady
invariant for a Fell algebra $A$. Given a Fell algebra $A$ with
spectrum $X$, choose any twist $\Gamma \to R$ such that $A$ is
Morita equivalent to $\tgcsa{\Gamma}{R}$. Applying $\partial^1$
to the class of $\Gamma$ in the twist group of $R$ yields an
element $\partial^1([\underline{\Gamma}])$ of $H^2(R,
\Tgerms)$. We show that the local homeomorphism $\psi :
\Gamma^{(0)} \to X$ obtained from the state-extension property
yields an isomorphism $\pi^*_\psi$ from the usual
sheaf-cohomology group $H^2(X, \Tgerms)$ to the
equivariant-sheaf cohomology group $H^2(R, \Tgerms)$. We then
show that the class $\delta(A) =
(\pi^*_\psi)^{-1}\big(\partial^1([\underline{\Gamma}])\big) \in
H^2(X, \Tgerms)$ does not depend on our choice of twist $\Gamma
\to R$, and regard $\delta(A)$ as an analogue of the
Dixmier-Douady invariant for $A$.

This paves the way for our main result, Theorem~\ref{thm:DD
classification}: Fell algebras $A_1$ and $A_2$ are Morita
equivalent if and only if there is a homeomorphism between
their spectra such that the induced isomorphism
$H^2\big(\widehat{A}_1, \Tgerms\big) \cong
H^2\big(\widehat{A}_2, \Tgerms\big)$
carries $\delta(A_1)$ to $\delta(A_2)$. The invariant is exhausted in
the sense that each element of $H^2(X, \Tgerms)$ can be
realised as $\delta(A)$ for some Fell algebra $A$ with spectrum
$X$ (Proposition~\ref{prop:invariant exhausted}).

A motivating example was a generalisation of Green's theorem
for free and proper transformation groups to free
transformation groups $(G, X)$ where $X$ is a Cartan $G$-space.
Our Corollary~\ref{cor-greens} gives a Morita equivalence
between the transformation-group $C^*$-algebra $C_0(X) \rtimes
G$ and the  $C^*$-algebra of the equivalence relation
induced by a local homeomorphism from a Hausdorff space $Y$ to
the (not necessarily Hausdorff) quotient space $G\backslash X$.
This result and its construction are prototypes for our later
investigations of diagonals in Fell algebras. In particular, we
show that $\delta(C_0(X) \rtimes G)$  is trivial.

\section{Preliminaries}\label{sec:prelim}

For a $C^*$-algebra $A$, let $\widetilde{A}$ denote the
$C^*$-algebra $A + \CC 1$ obtained by adjoining a unit.
If $B$ is a $C^*$-subalgebra of $A$, we regard $\widetilde{B}$
as a unital $C^*$-subalgebra of $\widetilde{A}$
(so, $1_{\widetilde{B}} = 1_{\widetilde{A}}$).

Given a Hilbert space $\Hh$, denote by $\Kk(\Hh)$ the
$C^*$-algebra of compact operators on $\Hh$.  For $\xi, \eta
\in \Hh$, let $\Theta_{\xi,\eta} \in \Kk(\Hh)$ be the rank-one
operator defined by
$\Theta_{\xi,\eta}(\zeta) = (\zeta |\eta)\xi$.

A $C^*$-algebra $A$ is \emph{liminary}
if $\pi(A) =  \Kk(\Hh_\pi)$ for every irreducible representation
$\pi$. If $B$ is an abelian $C^*$-algebra we freely identify $B$ and $C_0(\hat B)$.

Let $G$ be a Hausdorff topological groupoid with unit space
$G^{(0)}$. We denote the range and source maps by $r,s:G\to
G^{(0)}$ and the set of composable pairs of $G$ by $G^{(2)}$.
Let $U$ be a subset of the unit space. We write  $UG$, $GU$ and
$UGU$ for $r^{-1}(U)$, $s^{-1}(U)$ and $r^{-1}(U)\cap
s^{-1}(U)$; $U$ is called \emph{full} \label{def-full} if $s(UG)=G^{(0)}$. A
subset $T$ of $G$ is a \emph{$G$-set} if the
restrictions of $s$ and $r$ to $T$ are one-to-one. We
implicitly identify units of $G$ with the associated identity
morphisms throughout.

A groupoid is \emph{principal} if the map $\gamma\mapsto
(r(\gamma),s(\gamma))$ is one-to-one. A groupoid is
\emph{\'etale} if the range map (equivalently the source map)
is a local homeomorphism.  If $G$ is an \'etale groupoid then
the unit space $G^{(0)}$ is open in $G$ and for each $u\in
G^{(0)}$ the fibre $r^{-1}(u)$ is discrete.

A topological space $X$ is \emph{locally compact} if every point of $X$ has a compact neighbourhood in $X$; and $X$ is  \emph{locally Hausdorff} if every point of $X$ has a Hausdorff neighbourhood.

\section{Fell and Type~\texorpdfstring{$I_0$}{I0} algebras}\label{sec:Fell}
In this section we show that the classes of Fell and  Type~$I_0$ $C^*$-algebras
coincide.

Let $A$ be a $C^*$-algebra. A positive element $a$ of $A$ is
\emph{abelian} if the hereditary $C^*$-subalgebra
$\overline{aAa}$ generated by $a$ is commutative. If $A$ is
generated as a $C^*$-algebra by its abelian elements then $A$
is said to be of \emph{Type $I_0$} \cite[\S6.1]{Ped}.  An
irreducible representation $\pi_0$ of $A$ satisfies
\emph{Fell's condition}\label{pg-fell} if there exist $b\in A^+$ and an open
neighbourhood $U$ of $[\pi_0]$ in $\widehat{A}$ such that
$\pi(b)$ is a rank-one projection whenever $[\pi]\in U$; this
property goes back as far as \cite{Fell}. If every irreducible
representation of $A$ satisfies Fell's condition then $A$ is
said to be a \emph{Fell algebra} \cite[\S3]{AS}. That the Fell
algebras coincide with the Type $I_0$ $C^*$-algebras is a
consequence of the following lemma which is stated in
\cite[\S3]{AS}.

\begin{lemma}\label{lem-elementary?}
Let $A$ be a $C^*$-algebra and $\pi_0$ an irreducible
representation of $A$. Then there exists  an abelian element
$a$ of $A$ such that $\pi_0(a)\neq 0$ if and only if $\pi_0$
satisfies Fell's condition.
\end{lemma}
\begin{proof}
Suppose $a\in A^+$ is an abelian element such that $\pi_0(a)\neq
0$.  By rescaling we may assume that $\|\pi_0(a)\|=1$. By
\cite[Lemma~6.1.3]{Ped}, $\rk(\pi(a))\leq 1$ for all irreducible
representations $\pi$ of $A$.  Since $[\pi]\mapsto\|\pi(a)\|$ is
lower semicontinuous there exists a neighbourhood $U$ of
$[\pi_0]$ in $\hat A$ such that $\|\pi(a)\|>1/2$ when $[\pi]\in
U$. In particular, the spectrum $\sigma(\pi(a))$ of $\pi(a)$ is
$\{0,\lambda_\pi\}$ for some $\lambda_\pi>1/2$.  Fix $f\in
C([0,\|a\|])$ such that $f$ is identically zero on $[0,1/8]$ and
is identically one on $[1/4,\|a\|]$.  Set $b=f(a)$. If $[\pi]\in
U$ then
\[
\sigma(\pi(b))
    =\sigma(\pi(f(a)))
    =f(\sigma(\pi(a)))
    =f(\{0,\lambda_\pi\})
    =\{0,1\}.
\]
Since $\rk(\pi(a)) = 1$,  $\pi(b)$ is a rank-one
projection. Thus $\pi_0$ satisfies Fell's condition.

Now suppose that $\pi_0$ satisfies Fell's condition. Then there
exist $a\in A^+$ and an open neighbourhood $U$ of $[\pi_0]$ in
$\widehat{A}$ such that $\pi(a)$ is a  rank-one projection when
$[\pi]\in U$. Let $J$ be the closed ideal of $A$ such that
$\widehat{J}=U$.  There exists $x\in J^+$ such that
$\pi_0(axa)\neq 0$ (choose an approximate identity
$\{e_\lambda\}$ for $J$ and note that $\pi_0(e_\lambda)\to 1$
in $B(H)$).  Now $\pi(axa)=0$ whenever $[\pi]\not\in
\widehat{J}$, and $\rk(\pi(axa))\leq \rk(\pi(a))\leq 1$ when
$[\pi]\in\widehat{J}$. Thus $\rk(\pi(axa))\leq 1$ for all
irreducible representations of $A$, and  hence $axa$ is an
abelian element by \cite[Lemma~6.1.3]{Ped}.
\end{proof}

It is well known that if $p\in M(A)$ is a projection then $Ap$
is an $\overline{ApA}$--$pAp$-imprimitivity bimodule (see, for
example, \cite[Example~3.6]{tfb}).  More generally we have:

\begin{lemma}\label{lem-morita}
Let $A$ be a $C^*$-algebra. If $b\in A$ is self-adjoint\footnote{Added November 2011: unfortunately
our proof here assumes that $b=\sqrt{b^2}$, so it only works for positive $b$. But
Lemma~\ref{lem-morita} is true as stated.  To see that the left inner product is full we need to
show that $\overline{AbA}\subset \overline{Ab^2A}$. Since $\sqrt{|b|} \in \overline{Ab^2A}$ it
suffices to show that $b \in \overline{A\sqrt{|b|}A}$. Define $g(x)=\sqrt{|x|}$ for $x \in \RR$ and
define $f : \RR \to \RR$ by $f(x) = \sqrt{x}$ for $x \ge 0$ and $f(x) = -\sqrt{|x|}$ for $x < 0$.
Then $f(x)g(x) = x$ for all $x \in \RR$, so $b = f(b)\sqrt{|b|} \in A \sqrt{|b|} A$ as required.}
then $\overline{Ab}$ is an $\overline{AbA}$--$\overline{bAb}$-imprimitivity bimodule with actions
given by multiplication in $A$ and
\[
{}_{\overline{AbA}}\langle ab\,,\, cb\rangle=ab^2c^*\quad\text{and}\quad\langle ab\,,\, cb\rangle_{\overline{bAb}}=ba^*cb.
\]
\end{lemma}

\begin{proof}
The actions and inner products are restrictions of those on the
standard $A$-$A$-bimodule $A$, so we need only check that both
inner products are full.  The right inner product is full
because products are dense in $A$ and the left inner product is
full because $b \in C^*(\{b^2\}) \subset A$, so
${\overline{AbA}}={\overline{Ab^2A}}$.
\end{proof}

By \cite[Corollary~3.4]{AS}, the spectrum of a Fell algebra is
locally Hausdorff. So Fell algebras may be regarded as locally
continuous-trace $C^*$-algebras; since they are generated by
their abelian elements, they may also be regarded as locally
Morita equivalent to a commutative $C^*$-algebra. We make this
precise in the following theorem.


\begin{thm}\label{thm-fell}
Let $A$ be a $C^*$-algebra. The following are equivalent:
\begin{enumerate}
\item $A$ is of Type $I_0$;
\item there exists a collection $\{I_a : a \in S\}$ of
    ideals of $A$ such that $A$ is generated by these ideals
    and each $I_a$ is Morita equivalent to a commutative
    $C^*$-algebra;
\item $A$ is a Fell algebra.
\end{enumerate}
\end{thm}

\begin{proof}
((1) $\Longrightarrow$ (2)) Suppose $A$ is Type~$I_0$.  Let  $S$ be the set of
abelian elements of $A$.  For each $a\in S$, the sub-$C^*$-algebra
$\overline{aAa}$ is commutative, and by Lemma~\ref{lem-morita},
$\overline{aAa}$ is Morita equivalent to the ideal
$I_a:=\overline{AaA}$ generated by $a$. Since $A$ is generated
by $S$ it is also generated, as a $C^*$-algebra, by the
collection of ideals $\{I_a:a\in S\}$.

((2) $\Longrightarrow$ (3)) Assume~(2). Fix an irreducible
representation $\pi_0$ of $A$. Since $A$ is generated by the
ideals $I_a$ there exists $a_0$ such that $\pi_0$ does not
vanish on $I_{a_0}$. Morita equivalence preserves the property
of being a continuous-trace $C^*$-algebra, so $I_{a_0}$ is
itself a continuous-trace $C^*$-algebra. The restriction of
$\pi_0$ to $I_{a_0}$ is an irreducible representation
which satisfies Fell's condition in $I_{a_0}$ (since $I_{a_0}$
is a continuous-trace $C^*$-algebra). So there exist $b\in
I_{a_0}^+$ and a neighbourhood $U$ of $\widehat{I}_{a_0}$ such
that $\pi(b)$ is a rank-one projection whenever $[\pi]\in U$.
Now $b\in A^+$ and $U$ can be viewed as an open subset of $\hat
A$, so $\pi_0$ satisfies Fell's condition in $A$.

((3) $\Longrightarrow$ (1)) Suppose that $A$ is a Fell algebra.
By Lemma~\ref{lem-elementary?}, for each irreducible
representation $\pi$ of $A$ there exists an abelian element
$a_\pi \in A$ such that $\pi(a_\pi)\neq 0$. Let $B$ be the
$C^*$-algebra generated by the set $S$ of all abelian elements
of $A$, so that $B=\clsp\{a_1\cdots a_n:n\in\NN, a_i\in S\}$.
Since $\pi|_B\neq 0$ for all $\pi\in\widehat{A}$, $B$ is not
contained in any proper ideal of $A$.  But $B$ is an ideal of
$A$ by \cite[6.1.7]{Ped} (the largest Type $I_0$ ideal in fact),
so $B=A$ and $A$ is Type~$I_0$.
\end{proof}

%
%

\section{Green's theorem for free Cartan transformation groups.}\label{sec:green}

Throughout this section let $G$ be a second-countable, locally
compact, Hausdorff group acting continuously on a
second-countable, locally compact, Hausdorff space $X$. We will
prove a generalisation of Green's theorem for free group actions
which are not proper but only locally proper. Green's theorem
says that if a group $G$ acts freely and properly on a space
$X$, then the crossed product $C_0(X)\rtimes G$ is Morita
equivalent to $C_0(X/G)$; it follows that the Dixmier-Douady
invariant of the continuous-trace $C^*$-algebra $C_0(X)\rtimes
G$  is trivial. In section~\ref{sec:DD}, we will establish the
analogous result for locally proper actions and our
generalisation of the Dixmier-Douady classification.

Recall from \cite[Definition~1.1.2]{P} that $X$ is  a
\emph{Cartan $G$-space} if each point of $X$ has a wandering
neighbourhood $U$; that is, a neighbourhood $U$ such that
$\{s\in G:s\cdot U\cap U\neq \emptyset\}$ is relatively compact
in $G$. If $X$ is a Cartan $G$-space with a free action of $G$
then we will just say that $(G,X)$ is a free Cartan
transformation group.

The action of $G$ on $X$ is \emph{proper} if every compact
subset of $X$ is wandering.
Equivalently, the action is proper
if the map $\phi : G \times X \to X \times X$ given by
$\phi(g,x) = (g\cdot x, x)$ is proper in the sense that the
inverse images of compact sets are compact.
If $U$ is a
wandering neighbourhood in $X$, then the action of $G$ on the
saturation $G\cdot U$ of $U$ is proper by
\cite[Proposition~1.2.4]{P}.

If $G$ acts freely on $X$ and $x,y\in X$ with $G\cdot x=G\cdot
y$, then there is a unique $\tau(x,y)\in G$ such that
\begin{equation}\label{eq:tau def}
y=\tau(x,y)\cdot x;
\end{equation}
this defines a function $\tau$ from
\[
X\times_{G\backslash X} X:=\{(x,y)\in X\times X:G\cdot x=G\cdot y\}
\]
to $G$.  If $X$ is a free Cartan $G$-space, then $\tau$ is
continuous by \cite[Theorem~1.1.3]{P}.

The next lemma follows from \cite[I.10.1
Proposition~2]{Bourbaki}.

\begin{lemma}\label{lem-proper}
Suppose that the group $G$ acts freely on $X$. Then the action
of $G$ on $X$ is proper if and only if $\phi:G\times X\to
X\times X$, $(g,x)\mapsto (g\cdot x, x)$ is a homeomorphism
onto a closed subset of $X\times X$.
\end{lemma}
%
%


\begin{lemma}\label{lem-leftaction}
Suppose that $(G,X)$ is a free Cartan transformation group.
\begin{enumerate}
    \item There exists a covering  $\{U_i : i \in I\}$  of
        $X$ by $G$-invariant open sets such that $(G, U_i)$
        is proper for each $i$.
    \item Let $\{U_i : i \in I\}$ be a cover as in~(1), and
        let $W := \bigsqcup_i U_i$ be the topological
        disjoint union of the $U_i$. Then
     the map $\phi: G\times W \to W \times W$,
            $(g,x)\mapsto (g\cdot x, x)$ is a
            homeomorphism onto a closed subset of $W
            \times W$.
\end{enumerate}
\end{lemma}

\begin{proof}
(1) If $U$ is a wandering neighbourhood in $X$, then its
saturation $G\cdot U$ is a proper $G$-space by
\cite[Proposition~1.2.4]{P}.  So choose a cover $\{V_i : i \in
I\}$ of open wandering neighbourhoods in $X$ and then take
$U_i=G\cdot V_i$ for all $I$.

(2) The action of $G$ on $W$ is $g\cdot x^i=(g\cdot x)^i$,
where, for $x \in U_i \subset X$, we write $x^i$ for the
corresponding element in the copy of $U_i$ in $W$. The action
of $G$ on $W$ is free because the action of $G$ on $X$ is free.

Since the action on $W$ is continuous, so is $\phi$.
Since the action on $W$  is free, $\phi$ is
one-to-one.  The inverse $\phi^{-1}:\range\phi\to G\times
W$ is given by $\phi^{-1}(y,x)=(\tau(x,y),x)$. The
map $\tau:X\times_{G\backslash X} X\to G$ of~\eqref{eq:tau def}
is continuous because $(G,X)$ is Cartan, so $\phi^{-1}$ is
continuous. To see that the range of $\phi$ is closed, suppose
that
\[
    (g_n\cdot x_n^{i_n}, x_n^{i_n})
\]
is a sequence in $\range\phi$ converging to $(y,x^j)$.  Then
$x_n^{i_n}\to x^j$, so $i_n=j$ eventually. Since $U_j$ is
$G$-invariant, $g_n\cdot x_n^{i_n}\in U_j$ eventually as well.
Since $g_n\cdot x_n^{i_n}\to y$ it follows that $y\in U_j$. In
particular, $G\cdot x_n^j$ converges to both $G\cdot x^j$ and
$G\cdot y$. But the action of $G$ on $U_j$ is proper, so
$G\backslash X$ is Hausdorff and hence $y\in G\cdot x$ as
required.
\end{proof}

The following definitions are from \cite[\S2]{MRW87}. Let
$\Gamma$ be a locally compact, Hausdorff groupoid and $Z$ a
locally compact space. We say $\Gamma$ \emph{acts on the left of
$Z$} if there is a continuous open map $\rho:Z\to \Gamma^{(0)}$
and a continuous map $(\gamma,x)\mapsto \gamma\cdot x$ from
$\Gamma*Z=\Gamma \fibredprod{s}{\rho}
Z:=\{(\gamma,x)\in\Gamma\times Z : s(\gamma)=\rho(x)\}$ to $Z$
such that
\begin{enumerate}
\item $\rho(\gamma\cdot x)=r(\gamma)$ for $(\gamma,x)\in
    \Gamma \fibredprod{s}{\rho} Z$;
\item if $(\gamma_1,x)\in \Gamma \fibredprod{s}{\rho} Z$
    and $(\gamma_2,\gamma_1)\in\Gamma^{(2)}$ then
    $(\gamma_2\gamma_1)\cdot x=\gamma_2\cdot (\gamma_1\cdot
    x)$;
\item $\rho(x)\cdot x=x$ for $x\in Z$.
\end{enumerate}
Right actions of $\Gamma$ on $Z$ are defined similarly, except
that we use $\sigma:Z\to \Gamma^{(0)}$ and $Z
\fibredprod{\sigma}{r} \Gamma:=\{(x,\gamma)\in Z\times \Gamma :
\sigma(x)=r(\gamma)\}$. An action of $\Gamma$ on the left of $Z$
is said to be \emph{free} if $\gamma\cdot x=x$ implies that
$\gamma = \rho(x)$, and is said to be \emph{proper} if the map
$\Gamma \fibredprod{s}{\rho} Z\to Z\times Z:(\gamma,x)\mapsto
(\gamma\cdot x,x)$ is proper.

\begin{defn}\label{dfn:groupoid equivalence}
If $\Gamma_1$ and $\Gamma_2$ are groupoids then an
\emph{equivalence from $\Gamma_1$ to
$\Gamma_2$}\label{pg-equivalence}  is a triple $(Z, \rho,
\sigma)$ where
\begin{enumerate}
\item $Z$ carries a free and proper left-action of
    $\Gamma_1$ with fibre map $\rho : Z \to
    \Gamma_1^{(0)}$, and a free and proper right-action of
    $\Gamma_2$ with fibre map $\sigma : Z \to
    \Gamma_2^{(0)}$
\item the actions of $\Gamma_1$ and $\Gamma_2$ on $Z$
    commute, and
\item $\rho$ and $\sigma$ induce  bijections of $Z/\Gamma_2$ onto
    $\Gamma_1^{(0)}$ and of $\Gamma_1\backslash Z$ onto $\Gamma^{(0)}_2$, respectively.
\end{enumerate}
\end{defn}

Since $\rho$ and $\sigma$ are continuous open maps Definition
\ref{dfn:groupoid equivalence}(3) implies that $\rho$ and
$\sigma$  induce homeomorphisms $Z/\Gamma_2 \cong
\Gamma_1^{(0)}$ and $\Gamma_1\backslash Z\cong\Gamma^{(0)}_2$.
We will often just say that $Z$ is a
$\Gamma_1$--$\Gamma_2$-equivalence, leaving the fibre maps
$\sigma,\rho$ implicit. The main theorem of \cite{MRW87} says
that if $\Gamma_1$ and $\Gamma_2$ are groupoids with Haar
systems and $Z$ is a $\Gamma_1$--$\Gamma_2$-equivalence, then
$C_c(Z)$ can be completed to a
$C^*(\Gamma_1)$--$C^*(\Gamma_2)$-imprimitivity bimodule
\cite[Theorem~2.8]{MRW87}, so the full groupoid $C^*$-algebras
of $\Gamma_1$ and $\Gamma_2$ are Morita equivalent.

If $(G,X)$ is a transformation group we view $G\times X$ as a
transformation-group groupoid with composable elements
\[
(G\times X)^{(2)}=\big\{ ((g,x),(h,y))\in
(G\times X)\times (G\times X): x=h\cdot y\big\}
\]
and $(g,h\cdot y)(h,y)=(gh,y)$; the inverse is given by
$(g,x)^{-1}=(g^{-1},g\cdot x)$.  We identify the unit space
$(G\times X)^{(0)}=\{e\}\times X$ with $X$, so $s(g,x)=x$ and
$r(g,x)=g\cdot x$ for all $(g,x) \in G \times X$.

Suppose that $(G,X)$ is a free Cartan transformation group. Let
$\{U_i : i \in I\}$ be a covering of $X$ by $G$-invariant open
sets such that $(G, U_i)$ is proper for each $i$; then each
$V_i:=G\under U_i$ is locally compact and Hausdorff. Let $q:X\to
G\under X$ be the quotient map. For each $i$, denote by
$q_i:U_i\to V_i$ the restricted quotient map, and let
$\psi_i:V_i\to q(U_i)\subseteq G\under X$ be the inclusion
homeomorphism. Let $Y := \bigsqcup_i V_i$ be the topological
disjoint union of the $V_i$, and define $\psi:Y\to G\under X$ by
$\psi|_{V_i}=\psi_i$. Then $\psi$ is a local homeomorphism and
$Y$ is locally compact and Hausdorff.

\begin{lemma}\label{lem-disjoint}
Suppose that $X$ is a free Cartan $G$-space, and adopt the
notation of the preceding paragraph.
\begin{enumerate}
\item The formula $(g,x)\cdot (x,y)=(g\cdot x,y)$  defines
    a  free left action of the groupoid $G\times X$ on
    $X*Y=\{(x,y)\in X\times Y:q(x)=\psi(y)\}$.
\item The formula   $(g,x)\cdot x=g\cdot x$  defines a free
    and proper left action of the groupoid $G\times X$ on
    $W := \bigsqcup_i U_i$.
\item There is a homeomorphism $\alpha : W \to X*Y$ such
    that $(g,x)\cdot \alpha(z)  = \alpha((g,x)\cdot z)$ for
    all $g,x,z$.
\end{enumerate}
\end{lemma}

\begin{proof}
(1) Define $\rho: X*Y\to (G\times X)^{(0)}$ by $\rho(x,y)=x$.
Then
\[
    (G\times X)*(X*Y) = \{((g,x),(x',y)) : x=s(g,x)=\rho(x',y)=x'\}.
\]
It is straightforward to check that the formula
$(g,x)\cdot(x,y) := (g\cdot x, y)$ defines a free action of $G
\times X$ on the left of $X*Y$.

(2) As earlier, for $x \in U_i \subset X$, we write $x^i$ for
the corresponding element of $U_i \subset W$. Define
$\rho':W\to (G\times X)^{(0)}$ by $\rho'(x^i)=x$. Then
\[\textstyle
(G\times X)* W = \bigsqcup_i \{((g,x),x^i) : (g,x) \in G \times U_i\},
\]
and the formula $(g,x)\cdot x^i := (g\cdot x)^i$ defines a free
action of $(G \times X)$ on $W$.

By Lemma~\ref{lem-proper}, to see that the action is proper it suffices to verify that
\[\textstyle
    \phi:(G\times X)* W\to W \times W,
        \quad ((g,x),x^i)\mapsto((g\cdot x)^i, x^i)
\]
is a homeomorphism of $(G\times X) * W$ onto a closed subset of
$W \times W$.

Let $\tau : X \fibredprod{}{G \setminus X} X \to G$ be as
in~\eqref{eq:tau def}. Then $\tau$ is continuous since $X$ is a
Cartan $G$-space. So $\phi:(G\times X) * W \to \range\psi$ is
invertible with continuous inverse
\[
(y,x)\mapsto ((\tau(x,y),x),x).
\]
That the range of $\phi$ is closed is precisely
Lemma~\ref{lem-leftaction}(2).

(3) Define $\alpha: W \to X*Y$ by $\alpha(x^i)=(x,q_i(x^i))$.
Clearly $\alpha$ is continuous and one-to-one with
continuous inverse $(x, q_i(x^i))\mapsto x$.  To see that
$\alpha$ is onto, notice that $(x,y)\in X*Y$ for $y \in V_i$ if and only if
$y = q_i(x^i)$.
That $\alpha$ is equivariant is a simple calculation:
\[
(g,x)\cdot\alpha(x^i)
    =(g,x)\cdot (x^i, q_i(x^i))
    =(g\cdot x^i,  q_i(x^i))
    =\alpha(g\cdot x^i)
    =\alpha((g,x)\cdot x^i).\qedhere
\]
\end{proof}

Recall that under the relative topology
\[
    R(\psi)=\{(y_1, y_2)\in Y\times Y:\psi(y_1)=\psi(y_2)\}
\]
is a principal  groupoid with range and source maps $s(y_1, y_2)=y_2$,
$r(y_1, y_2)=y_1$, composition $(y_1, y_2)(y_2, y_3) = (y_1,
y_3)$ and inverses $(y_1,y_2)^{-1} = (y_2, y_1)$; $R(\psi)$ is \'etale because $\psi$ is a local homeomorphism.
We identify $R(\psi)^{(0)}$ with $Y$ via $(y,y) \mapsto y$.


\begin{thm}\label{thm-equiv} Let $(G,X)$ be a free Cartan
$G$-space. Then the transformation-group groupoid $G \times X$
is equivalent to the groupoid $R(\psi)$ described in the
preceding paragraph. More specifically, resume the notation of
Lemma~\ref{lem-disjoint} and define fibre maps $\rho : X*Y \to
(G \times X)^{(0)}$ and $\sigma : X*Y \to R(\psi)^{(0)}$ by
\[
\rho(x,y)=x\quad\text{ and }\quad \sigma(x,y) = y.
\]
Then the space $X*Y$ is a $(G\times X)$--$R(\psi)$ equivalence
under the actions
\begin{gather*}
(g,x)\cdot(x,y)=(g\cdot x, y)\quad\text{ and }\quad
(x,y)\cdot(y,z)=(x,z).
\end{gather*}
\end{thm}

\begin{proof}
We need to verify (1)--(3) of Definition~\ref{dfn:groupoid
equivalence}. By Lemma~\ref{lem-disjoint}, the left action of
$G \times X$ on $X*Y$ is free and proper.  It is easy to check
that the right action of $R(\psi)$ on $X*Y$ is free and proper,
verifying~(1).

To verify~(2), we calculate:
\[\begin{split}
\left( (g,x)\cdot(x,y) \right)\cdot (y,z)
    &=(g\cdot x,y)\cdot(y,z) \\
    &=(g\cdot x, z)=(g,x)\cdot(x,z)=(g,x)\cdot\left( (x,y) \cdot (y,z)\right).
\end{split}\]

It remains to verify~(3). Since both $\rho$ and $\sigma$ are
surjective, we need only show that both induce injections.
Suppose that $\rho(x,y)=\rho(x',y')$. Then certainly $x=x'$.
Since $(x,y),(x,y')\in X*Y$ we have $\psi(y)=q(x)=\psi(y')$, so
$(x,y)=(x, y')\cdot (y',y)$ with $(y,y')\in R(\psi)$. Hence
$\rho$ induces an injection.
Similarly, suppose $\sigma(x,y)=\sigma(x', y')$. Then $y=y'$.
Also, $q(x)=\psi(y)=q(x')$, so there exists $g\in G$ such that
$g\cdot x=x'$.  Thus $(g,x)\cdot(x,y)=(g\cdot x, y)=(x',y)$ and
$(g,x)\in G\times X$.  Hence, $\sigma$ induces an injection.
\end{proof}

We now obtain an analogue of Green's beautiful theorem for free
transformation groups:  if $G$ acts freely and properly on $X$
then $C_0(X)\rtimes G$ and $C_0(G\backslash X)$ are Morita
equivalent \cite[Theorem~14]{G}.  If the action is only locally
proper then $G\under X$ may not be Hausdorff, so that
$C_0(G\backslash X)$ is not a $C^*$-algebra --- the groupoid
$C^*$-algebra $C^*(R(\psi))$ serves as its replacement in this
case.

\begin{cor} \label{cor-greens}
Suppose that $(G,X)$ is a free Cartan transformation group. Then the
transformation-group $C^*$-algebra $C_0(X)\rtimes G$ is Morita
equivalent to the groupoid $C^*$-algebra $C^*(R(\psi))$.
\end{cor}
\begin{proof}
Since $R(\psi)$ is \'etale,  $R(\psi)$ has
a Haar systems given by counting measures. A natural Haar system for $G\times X$ is $\{\mu\times\delta_x:x\in X\}$, where $\mu$ is a left Haar measure on $G$ and $\delta_x$ is point-mass measure. So  the $(G\times X)$--$R(\psi)$
equivalence of Theorem~\ref{thm-equiv} induces a Morita
equivalence of full groupoid $C^*$-algebras by
\cite[Theorem~2.8]{MRW87}. Since $C^*(G\times X)$ and
$C_0(X)\rtimes G$ are isomorphic \cite[Remarks on p.
59]{Renault1980} the result follows.
\end{proof}

Let $(G,X)$ be a free Cartan transformation group. Then
$C_0(X)\rtimes G$ is a Fell algebra by \cite{aH}. Since the
property of being a Fell algebra is preserved under Morita
equivalence by \cite{aHRW}, $C^*(R(\psi))$ is also a Fell
algebra. Alternatively,  by \cite[Theorem~7.9]{C}  a
principal-groupoid $C^*$-algebra is a Fell algebra if and only
if the groupoid is Cartan in the sense that every unit has a
wandering neighbourhood (see Definition~7.3 of \cite{C}); it is
straightforward to verify the existence of wandering
neighbourhoods in $R(\psi)$.

\section{Fell algebras, the extension property and
\texorpdfstring{$C^*$}{C*}-diagonals}\label{sec:ext props}

In this section we show how to construct from a separable Fell
algebra $A$ a Morita equivalent $C^*$-algebra $C$ containing a
diagonal subalgebra in the sense of~\cite{Kumjian1986}. The
bulk of the work is to show that diagonal subalgebras of
separable Fell algebras can be characterised as those
abelian subalgebras which possess the extension property. We
start by verifying that the different notions of
diagonals in non-unital $C^*$-algebras which appear in the
literature coincide.

\subsection{Diagonals in nonunital \texorpdfstring{$C^*$}{C*}-algebras}

Let $A$ be a $C^*$-algebra and $B$ a $C^*$-subalgebra of $A$.
Recall that $P:A\to B$ is a \emph{conditional expectation} if
$P$ is a linear, norm-decreasing, positive map such that
$P|_B=\id_B$ and $P(ab)=P(a)b$, $P(ba)=bP(a)$ for all $a\in A$
and $b\in B$. We say $P$ is \emph{faithful} if $P(a^*a)=0$
implies $a=0$.

\begin{remark}\label{rem-cond-exp}
There are two other equivalent characterisations of a
conditional expectation:
\begin{enumerate}
\item $P:A\to B$ is a linear idempotent of norm 1;
\item $P:A\to B$ is a linear, norm-decreasing, completely
    positive map such that $P|_B=\id_B$ and $P(ab)=P(a)b$,
    $P(ba)=bP(a)$ for all $a\in A$ and $b\in B$.
\end{enumerate}
Our definition above implies (1); that (1) implies (2) is in \cite{tom} (see,  for example
\cite[Theorem~II.6.10.2]{bruce}),
and~(2) implies our definition since completely positive maps
are positive.
\end{remark}

\begin{defn}\label{dfn:diagonal}
Let $A$ be a separable $C^*$-algebra and let $B$ be an abelian
$C^*$-subalgebra of $A$. A \emph{normaliser} $n$ of $B$ in $A$
is an element $n\in A$ such that $n^* B n,  n B n^* \subset B$;
the collection of normalisers of $B$ is denoted by $N(B)$. A
normaliser $n$  is \emph{free} if $n^2 = 0$; the collection of
free normalisers of $B$ is denoted by $N_f(B)$. We  say  $B$ is
\emph{diagonal} or that $(A,B)$ is a \emph{diagonal pair} if
\begin{itemize}
\item[(D1)] $B$ contains an approximate identity for $A$;
\item[(D2)] there is a faithful conditional expectation $P
    : A \to B$; and
\item[(D3)] $\ker(P) = \clsp N_f(B)$.
\end{itemize}
\end{defn}
In~\cite[Definition~1.1]{Kumjian1986}, a pair $(A,B)$ of unital
$C^*$-algebras is said to be a diagonal pair if $1_B = 1_A$ and
(D2)~and~(D3) are satisfied, and a non-unital pair $(A,B)$ is
said to be a diagonal pair if the minimal unitisations form a
diagonal pair $(\widetilde{A}, \widetilde{B})$ (recall we
identify $1_{\widetilde{B}}$ with $1_{\widetilde{A}}$). If $A$
is unital then~(D1) implies that $B$ contains the unit of $A$,
so \cite[Definition~1.1]{Kumjian1986} and
Definition~\ref{dfn:diagonal} agree for unital $A$. We will use
the next two lemmas to show in Corollary~\ref{cor:diagonals are
diagonals}  that \cite[Definition~1.1]{Kumjian1986} and
Definition~\ref{dfn:diagonal} (which is the definition
implicitly used in \cite{Kumjian1985}) also coincide if $A$ is
nonunital.

\begin{lemma}\label{lem:expectation extends}
Let $A$ be a $C^*$-algebra with $C^*$-subalgebra $B$ and let
$P:A\to B$ be a conditional expectation.  Then $\widetilde
P:\widetilde A\to\widetilde B$ defined by $\widetilde
P((a,\lambda))=(P(a),\lambda)$ is also a conditional
expectation. Moreover, $P$ is faithful if and only if
$\widetilde P$ is.
\end{lemma}
\begin{proof}
Since $P:A\to B$ is a conditional expectation it is completely
positive by Remark~\ref{rem-cond-exp}. By
\cite[Lemma~3.9]{choi-effros}, $\widetilde P$ is also
completely positive, and the proof of
\cite[Lemma~3.9]{choi-effros} shows that $\widetilde P$ is
norm-decreasing. Since $\widetilde P(1_{\widetilde{A}}) =
1_{\widetilde{A}}$ and since $\widetilde P$ is idempotent, $\widetilde P$ is an idempotent of norm 1 and
hence is a conditional expectation by
Remark~\ref{rem-cond-exp}.

Now suppose that $P$ is faithful and that $(a,\lambda) \in
\widetilde{A}^+$ with $\widetilde P(a,\lambda)=0$. Since
$\widetilde P(a,\lambda) = (P(a), \lambda)$, we have $\lambda =
0$ and $P(a) = 0$.  Since $a \in A^+$ and $P$ is faithful, $a =
0$ also.  Hence, $\widetilde P$ is faithful. Conversely, if
$\widetilde P$ is faithful then so is its restriction $P$.
\end{proof}

\begin{lemma}\label{lem:exp kills normalisers}
Let $A$ be a $C^*$-algebra and $B$  an abelian $C^*$-subalgebra
of $A$.  Suppose that $B$ contains an approximate identity for
$A$. Then $n^*n \in B$ for all $n \in N(B)$. If, in addition,
$P : A \to B$ is a conditional expectation, then $P(n) = 0$
for all $n \in N_f(B)$.
\end{lemma}
\begin{proof}
Fix $n \in N(B)$ and let  $(b_i)_{i \in I}$ be an approximate
identity for $A$ contained in $B$.   Then we have $n^*n =
\lim_{i \in I} n^* b_i n \in B$.

Now fix $n \in N_f(B)$. Set $a_k = (n^*n)^{1/k}$. A standard
spatial argument using the polar decomposition of $n$ shows that
$na_k \to n$. To see that $P(n)=0$, it suffices by continuity to
show that $P(na_k) = 0$ for all $k$. Fix $k$. Then
\[
P(na_k) = P(n)a_k = a_kP(n) = P(a_kn)
\]
since  $a_k \in B$ and $P$ is a conditional expectation. Since
$n \in N_f(B)$, we have $(n^*n) n = n^* n^2 = 0$ and it follows
that $a_kn = (n^*n)^{1/k} n = 0$. Hence, $P(na_k) = 0$ as
required.
\end{proof}

\begin{lemma}\label{lem-kerneloftildeP}
Suppose that $(A, B)$ is a diagonal pair with  expectation
$P:A\to B$. Let $\widetilde{P} : \widetilde{A} \to
\widetilde{B}$ be the conditional expectation of
Lemma~\ref{lem:expectation extends}. Then
\[
N_f(\widetilde B) = \{(n,0): n\in N_f(B)\} \ \text{and} \
\ker\widetilde{P} = \clsp N_f(\widetilde B).
\]
\end{lemma}

\begin{proof}
Fix $n \in N_f(B)$ and $(b, \mu) \in \widetilde B$. Then $(n,
0)^2 = 0$ and
\[
(n^*, 0)(b, \mu)(n,0) = (n^*bn + \mu n^*n, 0) \in \widetilde B
\]
by Lemma~\ref{lem:exp kills normalisers}. Similarly, $(n, 0)(b, \mu)(n^*, 0)  \in \widetilde B$.
Hence, $(n, 0) \in N_f(\widetilde B)$, giving $\{(n,0) : n \in
N_f(B)\} \subset N_f(\widetilde{B})$.
Now fix $c = (n, \lambda) \in N_f(\widetilde B)$. Since $(n^2 +
2\lambda n, \lambda^2) = c^2 = 0$, we have $\lambda = 0$ and
$n^2 = 0$. We now verify that $n$ normalises $B$.  Fix $b \in
B$.  Then since $(n, 0) \in N_f(\widetilde B)$ and $(b, 0) \in
\widetilde B$ we have
\[
(n^*bn, 0) = (n, 0)^*(b, 0)(n, 0) \in \widetilde B.
\]
Hence $n^*bn \in B$. Similarly, $nbn^* \in B$.
This proves $N_f(\widetilde B) = \{(n,0): n\in N_f(B)\}$.

Since  $(A,B)$ is a diagonal pair, we have $\ker P = \clsp
N_f(B)$.  Hence,
\[
\ker\widetilde P =\{(a,0):a\in\ker P\} = \clsp\{(n,0):  n\in N_f(B)\} = \clsp N_f(\widetilde B).\qedhere
\]
\end{proof}

\begin{cor}\label{cor:diagonals are diagonals}
Let $A$ be a nonunital $C^*$-algebra and let $B$ be an abelian
$C^*$-subalgebra of $A$. Then $(A,B)$ is a diagonal pair in the
sense of Definition~\ref{dfn:diagonal} if and only if
$(\widetilde{A},\widetilde{B})$ is a diagonal pair in the sense
of \cite[Definition~1.1]{Kumjian1986}.
\end{cor}
\begin{proof}
First suppose that $(A,B)$ is diagonal with conditional
expectation $P : A \to B$. We have $1_{\widetilde A}\in
\widetilde B$ by definition of the inclusion of $\widetilde{B}$
in $\widetilde{A}$. Lemma~\ref{lem:expectation extends} implies
that $\widetilde{P} :\widetilde{A} \to \widetilde{B}$ is
faithful.  Moreover, by Lemma~\ref{lem-kerneloftildeP} we have
$\ker \widetilde{P} = \clsp N_f(\widetilde B)$. Thus
$(\widetilde A, \widetilde B)$ is a diagonal pair  in the sense
of \cite[Definition~1.1]{Kumjian1986}.

Conversely, suppose $(\widetilde{A}, \widetilde{B})$ is a
diagonal pair, in the sense of
\cite[Definition~1.1]{Kumjian1986}, with conditional expectation
$Q:\widetilde A\to\widetilde B$. Since $Q$ is faithful, $P :=
Q|_A$ is also a faithful conditional expectation, and $Q =
\widetilde{P}$.

As in the proof of Lemma~\ref{lem-kerneloftildeP} if $(n, \lambda) \in
N_f(\widetilde B)$, then $\lambda = 0$ and $n \in N_f(B)$. So
\[
    N'_f(B) := \{ n \in A : (n, 0) \in  N_f(\widetilde B) \} \subset N_f(B).
\]
By assumption $\ker \widetilde{P} = \clsp N_f(\widetilde B)$.
By definition of $\widetilde{P}$, we have $\ker \widetilde{P} =
\{ (a, 0) : a \in \ker P \}$. Hence $\ker P = \clsp
N'_f(B)\subset\clsp N_f(B)$.

Fix an approximate identity $(b_i)_{i \in I}$ for $B$; we claim
it is also an approximate identity for $A$. Since $A = B + \ker
P$ and $\ker P = \clsp N'_f(B)$, it suffices to show that $nb_i
\to n$ for each $n \in N'_f(B)$. Fix $n \in N'_f(B)$. Since
$(n, 0) \in N_f(\widetilde B)$, we have $(n^*, 0)(0, 1)(n, 0) =
(n^*n, 0) \in \widetilde B$, so $n^*n \in B$. Since $n^*nb_i
\to n^*n$, it follows that $nb_i \to n$ also, so $(b_i)_{i \in
I}$ is an approximate identity for $A$.

Lemma~\ref{lem:exp kills
normalisers} now gives  $\clsp
N_f(B)\subset \ker P$, and hence $\ker P=\clsp
N_f(B)$.
\end{proof}

Corollary~\ref{cor:diagonals are diagonals} above ensures, in
particular, that we may apply the results of~\cite{Kumjian1986}
to our diagonal pairs, and we shall do so without further
comment.

\subsection{Diagonals in Fell algebras and the extension
property}

Building on the seminal work of Kadsion and Singer
\cite{Kadison-Singer-1959}, Anderson defined the extension
property for a pair of unital $C^*$-algebras
\cite[Definition~3.3]{Anderson-1979}  as follows. Let $A$ be a
unital $C^*$-algebra and $B$ a $C^*$-subalgebra with $1_A \in
B$. Then $B$ is said to have the extension property relative to
$A$ if each pure state of $B$ has a unique extension to a pure
state of $A$ (equivalently, each pure state of $B$ has a unique
extension to a state of $A$ --- this extension is then
necessarily pure). If $B$ is abelian and has the extension
property relative to $A$ then $B$ must be maximal abelian by the Stone-Weierstass Theorem \cite[p.~311]{Anderson-1979}. The
converse is false: for example, Cuntz has shown that the
canonical maximal abelian subalgebra of $\mathcal{O}_n$ does not
have the extension property \cite[Proposition 3.1]{Cuntz1980}; the next example shows this can happen even in a Fell
algebra.

\begin{example}\label{ex-counter}
Let $B=C([-1,1])$ and let $G=\{0,1\}$ act on $[-1,1]$ by $g
\cdot x = (-1)^g x$. Then the crossed product $A=B\rtimes G$ is
generated by $B$ and a self-adjoint unitary $U$ which does not
commute with $B$. That $A$ is a Fell algebra follows from, for
example, \cite[Lemma~5.10]{aHintegrable}. Moreover, $B$ is a
maximal abelian subalgebra of $A$. By
\cite[Theorem~5.3]{dana-ccr} the spectrum of $A$ is homeomorphic
to $\{\pi_{-1},\pi_1\}\cup(0,1]$ where $t_n\to\pi_1,\pi_{-1}$
for $t_n\in (0,1]$ if and only if $t_n\to 0$ in $\RR$. In
particular, $\pi_{-1}$ and $\pi_{1}$ cannot be separated by
disjoint open sets. The $\pi_i$ are one-dimensional
representations determined by $\pi_j(f) = f(0)$ for $f \in B$
and $\pi_j(U) = j$. Hence $\pi_1, \pi_{-1}$ are distinct pure
states of $A$ which restrict to evaluation at 0 on $B$. Thus $B$
is a maximal abelian subalgebra but does not have the extension
property.
\end{example}

Let $A$ be a unital $C^*$-algebra, and let be $B$ be a maximal
abelian subalgebra of $A$. Then $B$ has the extension property
relative to $A$ if and only if there exists a conditional
expectation $P:A\to B$ such that for each pure state $h$ of $B$,
the state $h\circ P$ is its unique pure state extension to $A$
\cite[Theorem~3.4]{Anderson-1979}. By \cite[Theorem~2.4]{ABG}
$B$, whether or not it is maximal abelian, has the extension
property relative to $A$ if and only if $A = B + \clsp[B, A]$
where $[B, A] = \{ ba - ab : a \in A,\ b \in B \}$. The
techniques used in the proof imply that the extension property
is equivalent to the requirement that $B + \lsp[B, A]$ be dense
in $A$ (if $f$ is a state on $A$ which restricts to a pure state
on $B$, then $f(ab) = f(a)f(b) = f(ba)$ for all $a \in A$, $b
\in B$ and hence $f$ vanishes on $\lsp[B, A]$).

We use the following definition of the extension property for
non-unital $C^*$-algebras.

\begin{defn}\label{dfn:ext prop} %
Let $B$ be a $C^*$-subalgebra of a $C^*$-algebra $A$. As in
\cite[\S2]{Kumjian1985}, we say that \emph{$B$ has the
extension property relative to $A$} if
\begin{enumerate}
\item $B$ contains an approximate identity for $A$; and
\item every pure state of $B$ extends uniquely to a pure
    state of $A$.
\end{enumerate}
\end{defn}

By \cite[Proposition~1.4]{Kumjian1986}, if $(A, B)$ is a
diagonal pair, then $B$ has the extension property relative to
$A$.

\begin{remark}\label{rmk:nonunital}
\begin{enumerate}
\item The extension property as presented in
    \cite[Definition~2.5]{ABG} seems slightly different to
    Definition~\ref{dfn:ext prop}: in the former $B$ is
    said to have the extension property relative to $A$ if
    pure states of $B$ extend uniquely to pure states of
    $A$ and no pure state of $A$ annihilates $B$. As noted
    in \cite[\S 2]{Wassermann2008} these two definitions
    are equivalent: it follows from \cite[Lemma
    2.32]{Akemann-Shultz} that $B$  contains an approximate
    identity for $A$ if and only if no pure state of $A$
    annihilates $B$.

\item Let $B$ be an abelian $C^*$-subalgebra of a nonunital
    $C^*$-algebra $A$. By \cite[Remark~2.6(iii)]{ABG} $B$
    has the extension property relative to $A$ if and only
    if $\widetilde{B}$ has the extension property relative
    to $\widetilde{A}$ (and $B$ is maximal abelian in $A$
    if and only if $\widetilde{B}$ is maximal abelian in
    $\widetilde{A}$). Moreover,  as in the unital case, $B$
    has the extension property relative to $A$ if and only
    if $B + \lsp[B, A]$ is dense in $A$.
\end{enumerate}
\end{remark}

\begin{notation+h}\label{rmk:def-psi}
Let $B$ be an abelian $C^*$-subalgebra of a $C^*$-algebra $A$,
and suppose that $B$ has the extension property relative to $A$.
By the discussion above, $B$ is maximal abelian and there exists
a unique conditional expectation $P:A \to B$. Moreover, for each
pure state $h$ of $B$, the state  $h\circ P$ is its unique pure
state extension to $A$. For this reason, we say that the
extension property is \emph{implemented by P}. The map $x
\mapsto x\circ P$ is a weak*-continuous map from the set of pure
states of $B$ (which may be identified with $\widehat{B}$) to
the pure states of $A$.

Of course $x \circ P$  determines a GNS triple $(\pi_x,
\Hh_x, \xi_x)$. That is, $\pi_{x}$ is an irreducible
representation of $A$ on the Hilbert space $\Hh_x$, the unit
vector $\xi_x$ is cyclic vector for $\pi_x$, and $x\circ P(a) =
(\pi_{x}(a)\xi_x\,|\, \xi_x)$ for all $a \in A$. Let $\psi =
\psi_P : \widehat{B} \to \widehat{A}$ be the map which takes $x
\in \widehat{B}$ to the unitary equivalence class $[\pi_{x}] \in
\widehat{A}$.  We call $\psi$ the \emph{spectral map} associated
to the inclusion $B \subset A$.

Since diagonal pairs have the extension property, it follows
from the above that if $(A,B)$ is a diagonal pair, then the
conditional expectation from $B$ to $A$ is unique. We use this
frequently: given a diagonal pair $(A,B)$, we will
without comment refer to \emph{the} expectation $P:B\to A$ and
use that the extension property is implemented by $x
\mapsto x \circ P$.
\end{notation+h}

There is some overlap with Lemma~\ref{lem:D-points} and \cite[Proposition~2.10]{A}.

\begin{lemma}\label{lem:D-points}
Suppose that $A$ is a separable $C^*$-algebra, let $B$ be an
abelian $C^*$-subalgebra with the extension property relative
to $A$ implemented by $P:A\to B$, and let $\psi : \widehat{B}
\to \widehat{A}$ be the spectral map. Suppose that $\pi$ is an
irreducible representation of $A$ such that
$\pi(A)=\Kk(H_\pi)$. Then $\psi^{-1}(\{[\pi]\})$ is a discrete
countable subset of $\widehat{B}$, and there exist a listing
$\{x_\lambda : \lambda \in \Lambda\}$ of $\psi^{-1}(\{[\pi]\})$
and a basis $\{\xi_\lambda : \lambda \in \Lambda\}$ of $H_\pi$
such that $x_\lambda\circ P =
(\pi(\cdot)\xi_\lambda\,|\,\xi_\lambda)$ for all $\lambda \in
\Lambda$ and
\begin{equation}\label{eq:basis decomp}
\pi(b) = \sum_{\lambda \in \Lambda} x_\lambda(b) \Theta_{\xi_\lambda,\xi_\lambda}
     \quad\text{ for all $b \in B$.}
\end{equation}
Furthermore, if $A$ is liminary, then $\psi$ is surjective and
$P: A \to B$ is faithful.
\end{lemma}
\begin{proof}
We begin by identifying a basis $\{\xi_\lambda : \lambda \in
\Lambda\}$ for $H_\pi$ and points $\{x_\lambda : \lambda \in
\Lambda\}$ in $\widehat{B}$ satisfying~\eqref{eq:basis decomp}.
We then show that the $x_\lambda$ form a discrete set which
coincides with $\psi^{-1}(\{[\pi]\})$.

We have $\pi(B)$ maximal abelian in $\pi(A)$ by
\cite[Corollary~3.2]{ABG}. Since $\pi(A)=\Kk(H_\pi)$, we have
$B/\ker\pi\cong \pi(B) =
\clsp\{\Theta_{\xi_\lambda,\xi_\lambda} : \lambda \in
\Lambda\}$ for some orthonormal basis $\{\xi_\lambda : \lambda \in
\Lambda\}$ of $\Hh_\pi$
\footnote{This standard fact follows from the Spectral Theorem. Specifically, the Spectral Theorem implies firstly that the $C^*$-algebra generated by each self-adjoint $T \in \Kk(H)$ is equal to the $C^*$-algebra generated by its spectral projections. So any abelian $C^*$-subalgebra $D$ of $\Kk(H)$ is spanned by commuting finite-dimensional projections. A minimal subprojection of any of these spanning projections then commutes with $D$. So if $D$ is maximal abelian, then it is spanned by a maximal family of mutually orthogonal minimal projections on $H$.}%
; and $\Lambda$ is countable because $A$
is separable. Since each one-dimensional subspace $\lsp\{\xi_\lambda\}$ is
invariant under $\pi(B)$, it determines an irreducible
representation of $B$ given by point evaluation at
$x_\lambda \in \widehat{B}$. The set $\{x_\lambda : \lambda \in
\Lambda\}$ is discrete because for each $\lambda$ there exists
$b_\lambda$ such that $\pi(b_\lambda) =
\Theta_{\xi_\lambda,\xi_\lambda}$ which forces
$x_\mu(b_\lambda) = 0$ for $\lambda \not= \mu$. The formula~\eqref{eq:basis decomp}
follows from the definition of the $x_\lambda$.

Fix $\lambda \in \Lambda$. Then for $b \in B$,
\[
(\pi(b)\xi_\lambda\,|\,\xi_\lambda)=
\Big(\sum_{\mu \in \Lambda} x_\mu(b)\Theta_{\xi_\mu,\xi_\mu}(\xi_\lambda)\,|\,\xi_\lambda\Big)
=x_\lambda(b)=x_\lambda\circ P(b).
\]
Hence $x_\lambda\circ P=
(\pi(\cdot)\xi_\lambda\,|\,\xi_\lambda)$ for all $\lambda \in
\Lambda$ by the extension property. Thus
$\psi(x_\lambda) = [\pi]$, and it follows that $\{x_\lambda :
\lambda \in \Lambda\}\subset\psi^{-1}(\{[\pi]\})$.

For the other inclusion, let $x\in \psi^{-1}(\{[\pi]\})$. Since the GNS representation associated to $x\circ P$ is equivalent to $\pi$, we may assume that $x\circ P(\cdot)=(\pi(\cdot)\xi\,|\,\xi)$ for some unit vector $\xi\in H_\pi$. Using~\eqref{eq:basis decomp} we get
\[
x(b)=\sum_{\lambda\in\Lambda} x_\lambda(b)\, |(\xi\,|\, \xi_\lambda)|^2\quad\text{for all $b\in B$.}
\]
Suppose that there exist $\lambda_i$ such that $(\xi\,|\, \xi_{\lambda_i})\neq 0$ for $i=1,2$.  Since $\{x_\lambda : \lambda \in
\Lambda\}$ is discrete we can find $b_i\in B$ such that $x_\lambda(b_i)=0$ unless $\lambda=\lambda_i$.  Now $x(b_1b_2)=0$ but
\[x(b_1)x(b_2)=x_{\lambda_1}(b_1)|(\xi\,|\, \xi_{\lambda_1})|^2x_{\lambda_2}(b_2)|(\xi\,|\, \xi_{\lambda_2})|^2\neq 0
\]
which is impossible.  It follows that there is precisely one $\lambda$ such that $(\xi\,|\, \xi_{\lambda})\neq 0$, and hence that $x=x_\lambda$. Thus $\{x_\lambda : \lambda \in \Lambda\}=\psi^{-1}(\{[\pi]\})$.

Now suppose that $A$ is liminary and let $\pi$ be an irreducible
representation of $A$.  Then $\pi(A)=\Kk(H_\pi)$, so the above
argument shows that $\psi^{-1}(\{[\pi]\})$ is nonempty.
Therefore, $\psi$ is surjective. It remains to prove that $P$ is
faithful.  Fix $a \in A^+ \setminus \{0\}$. There is an
irreducible representation $\pi$ on a Hilbert space $\Hh_\pi$
with $\pi(a) \ne 0$.   Then with $\{\xi_\lambda : \lambda \in
\Lambda\}$ and $\psi^{-1}(\{[\pi]\}) = \{x_\lambda : \lambda \in
\Lambda\}$ as in the statement of the lemma we have
\[
\pi(P(a)) =  \sum_{\lambda \in \Lambda} x_\lambda(P(a)) \Theta_{\xi_\lambda,\xi_\lambda}
=  \sum_{\lambda \in \Lambda} (\pi(a) \xi_\lambda\,|\, \xi_\lambda)
                                \Theta_{\xi_\lambda,\xi_\lambda} \ne 0.
\]
Hence $P(a) \ne 0$ and $P$ is faithful.
\end{proof}

\begin{lemma}\label{lem-claim1}
Let $A$ be a separable liminary $C^*$-algebra and $B$ an
abelian $C^*$-subalgebra with the  extension property relative
to $A$, and let $\psi$ be the spectral map. Let $U$ be an open
subset of $\widehat B$ and let $J = \{b\in B: y(b)=0\text{\ for
all\ }y\in\widehat B\setminus U\} \lhd B$. Let $I =
\overline{AJA}$ be the ideal of $A$ generated by $J$. Then
\begin{equation}\label{eq-open}
\widehat I=\{[\pi]\in\widehat A:\pi|_J\neq 0\}=\psi(U).
\end{equation}
\end{lemma}

\begin{proof}
Since $I$ is generated by $J$, we have
\[
I   =\bigcap\{\ker\pi:[\pi]\in\widehat A, I\subseteq\ker\pi\}
    =\bigcap\{\ker\pi:[\pi]\in\widehat A, J\subseteq\ker\pi\},
\]
which gives $\widehat I=\{[\pi]\in\widehat A:\pi|_J\neq 0\}$.

To prove that $\psi(U) \subset \{[\pi]\in\widehat A:\pi|_J\neq
0\}$, let $P$ be the unique conditional expectation from $A$ to
$B$. Fix $x \in U$,
and let $\pi\in\psi(x)$. Since $x \in U=\widehat J$, there is
an element $b \in J$ such that $x(b) \neq 0$. Since $x\circ P$ is a pure state associated with $\pi$ there is a unit
vector $\xi\in\Hh_\pi$ such that $x\circ P(a) =
(\pi(a)\xi\,|\,\xi)$ for all $a \in A$.    But $x\circ P(b)=
x(b)\neq 0$, so $\pi(b)\neq 0$.  Hence, $\psi(U) \subset
\{[\pi]\in\widehat A:\pi|_J\neq 0\}$.

To see that $\{[\pi]\in\widehat A:\pi|_J\neq 0\}\subset
\psi(U)$, fix an irreducible representation $\pi$ of $A$ with
$\pi(f)\neq 0$ for some $f\in J$.  Since $A$ is liminary,
$\pi(A)=\Kk(H_\pi)$ so Lemma~\ref{lem:D-points} implies that
$\psi^{-1}([\pi])$ is a countable discrete set $\{x_\lambda :
\lambda\in \Lambda\} \subset \widehat{B}$, and there is a basis
$\{\xi_\lambda : \lambda \in \Lambda\}$ for $\Hh_\pi$ such that
$\pi(b)=\sum_{\lambda \in \Lambda}
x_\lambda(b)\Theta_{\xi_\lambda,\xi_\lambda}$ for all $b \in
B$. Since $\pi(f)\neq 0$ and $\widehat J=U$, there exists
$\lambda \in \Lambda$ such that $x_\lambda\in U$ and
$f(x_\lambda)\neq 0$. Thus $[\pi]=\psi(x_\lambda)\in \psi(U)$.
Hence $\{[\pi]\in\widehat A:\pi|_J\neq 0\}= \psi(U)$.
\end{proof}

The following lemma is used implicitly in the proof of
\cite[Theorem~3.1]{Kumjian1986}.

\begin{lemma}\label{lem:normaliser fairy}
Let $A$ be a separable liminary $C^*$-algebra and $B$ an
abelian $C^*$-subalgebra with the  extension property relative
to $A$. Let $\psi$ be the spectral map.
\begin{enumerate}
\item Suppose that $f,g \in B^+$ have the property that the
    restriction of $\psi$ to $\supp f\cup\supp g$ is
injective. Then $gAf \subset B$.
\item If $f,g \in B^+$ have the property that the
    restrictions of $\psi$ to $\supp f$ and $\supp g$ are
    injective, then $gAf \subset N(B)$.
\end{enumerate}
\end{lemma}

\begin{proof}
Let $P : A \to B$ be the unique conditional expectation.

(1) Fix $a \in A$ and an irreducible representation $\pi : A
\to \Bb(\Hh_\pi)$.  It suffices to show that
\begin{equation}\label{eq:P fixes fcf}
\pi(P(gaf)) = \pi(gaf).
\end{equation}
Since $P$ is an expectation with $f, g\in P(A)$, we have
$P(gaf) = g P(a) f$, so~\eqref{eq:P fixes fcf} is trivial if
$\pi(f) = 0$ or $\pi(g)=0$. So we suppose that
$\pi(f),\pi(g) \not= 0$ and we verify that $\pi(gaf) = \pi(
gP(a) f)$.

Since $A$ is liminary, we may use Lemma~\ref{lem:D-points} to
obtain a listing $\psi^{-1}(\{[\pi]\}) = \{x_\lambda : \lambda
\in \Lambda\}$ and a basis $\{\xi_\lambda : \lambda \in
\Lambda\}$ of $\Hh_\pi$ such that $x_\lambda\circ P =
(\pi(\cdot)\xi_\lambda\,|\,\xi_\lambda)$ for all $\lambda \in
\Lambda$ and $\pi(b) = \sum_{\lambda \in \Lambda}
x_\lambda(b)\Theta_{\xi_\lambda,\xi_\lambda}$ for all $b \in
B$. Since $\psi(x_\lambda) = [\pi]$ for all $\lambda$, and
since $\psi$ restricts to an injection on $\supp f\cup\supp g$
there exists a unique $\lambda \in \Lambda$ such that
$x_\lambda\in \supp f\cup\supp g$. Thus $\pi(f) = x_\lambda(f)
\Theta_{\xi_\lambda,\xi_\lambda}$ and $\pi(g) = x_\lambda(g)
\Theta_{\xi_\lambda,\xi_\lambda}$. Hence
\begin{align*}
\pi(gaf)
    &=  (x_\lambda(g) \Theta_{\xi_\lambda,\xi_\lambda}) \pi(a)
            (x_\lambda(f) \Theta_{\xi_\lambda,\xi_\lambda})
    =  x_\lambda(g) (\pi(a)\xi_\lambda\,|\,\xi_\lambda)
        x_\lambda(f)\Theta_{\xi_\lambda,\xi_\lambda} \\
    &=  x_\lambda(g) x_\lambda(P(a))x_\lambda(f) \Theta_{\xi_\lambda,\xi_\lambda}
    = \pi(gP(a)f).
\end{align*}
So $gaf=P(gaf)$ and hence $gAf \subset B$.

\smallskip

(2)  Fix $a \in A$ and set $n := gaf$.  Then for every $b \in
B$ we have $n^*bn = f(a^*gbga)f \in B$ by~(1). Thus, $n^*Bn
\subset B$ and symmetrically  $nBn^* \subset B$.  
Hence,  $n = gaf \in N(B)$.
\end{proof}

Our next result, Theorem~\ref{thm:ext prop => diag}, extends
\cite[Theorem 2.2]{Kumjian1985} from continuous-trace
$C^*$-algebras to Fell algebras;  indeed our proof follows
similar lines. There is also some overlap with
\cite[Proposition~3.3]{A} and \cite[Proposition~4.1]{BC}.
Example~\ref{ex:alex} below shows that Theorem~\ref{thm:ext
prop => diag} cannot be extended to bounded-trace
$C^*$-algebras.

\begin{thm}\label{thm:ext prop => diag}
Let $A$ be a separable Fell algebra and let $B$ be an abelian
$C^*$-subalgebra with the  extension property relative to $A$.
Then
\begin{enumerate}
\item The spectral map $\psi$  is a local homeomorphism,
    and
\item $(A,B)$ is a diagonal pair.
\end{enumerate}
\end{thm}
\begin{proof}
(1) We must prove that $\psi$ is continuous, open, surjective
and locally injective. Continuity follows from the observation
that $\phi \mapsto \phi \circ P$ is a weak$^*$-continuous map
from the state space of $B$ to that of $A$. That $\psi$ is an
open map follows from Lemma~\ref{lem-claim1} and the
surjectivity of $\psi$ follows from Lemma~\ref{lem:D-points}.

To show that $\psi$ is locally injective we argue as in
\cite[Theorem~2]{Kumjian1988}. Suppose that $\psi$ fails to be
locally injective at $x\in\widehat B$.   Then there exist
sequences $(y_n)^\infty_{n=1}$, $(z_n)^\infty_{n=1}$ in
$\widehat B$ such that $y_n, z_n \to x$ and, for all $n$,
$y_n\neq z_n$ and $\psi(y_n)=\psi(z_n)$.  Let $\pi\in\psi(x)$.
Since $A$ is a Fell algebra there exists a Hausdorff
neighbourhood $V$ of $[\pi]$ in $\widehat A$
\cite[Corollary~3.4]{AS}. Let $I$ be the ideal of $A$ such that
$\widehat I=V$. Since $V$ is Hausdorff, $I$ is a continuous
trace $C^*$-algebra. Let $U=\psi^{-1}(V)$ and let $J$ be the
ideal of $B$ such that $\widehat J=U$.  We have $J\subset I$ by
Lemma~\ref{lem-claim1}.

By Lemma~\ref{lem:D-points}, $\psi^{-1}(\{[\pi]\}) = \{x_\lambda
: \lambda \in \Lambda\}$ is discrete and countable, and there
exists a basis $\{\xi_\lambda : \lambda \in \Lambda\}$ of
$\Hh_\pi$ such that $\pi(b)=\sum_{\lambda \in \Lambda}
x_\lambda(b)\Theta_{\xi_\lambda,\xi_\lambda}$ for all $b \in B$.
Choose $b\in C_c(\widehat B)^+$ such that $x(b)>0$ and $\supp
b\subset U$. Since $\pi \in \psi(x)$, we have $x = x_\mu$ for
some $\mu \in \Lambda$, and since $\psi^{-1}(\{[\pi]\})$ is
discrete, we may choose $g\in C_c(\widehat B)^+$ such that
$x_\lambda(g)=0$ unless $\lambda = \mu$. So $\pi(bg)$ is a
positive multiple of $\Theta_{\xi_\mu,\xi_\mu}$, so $f :=
\frac{1}{x_\mu(bg)} bg \in C_c(\widehat B)^+$ satisfies
$\pi(f)=\Theta_{\xi_\mu,\xi_\mu}$ and $x(f)=1$.

We have $\supp(f) \subset \supp(b) \subset U$, so $f \in J
\subset I$. Since $f$ has compact support it belongs to the
Pedersen ideal of $I$ and hence is a continuous trace element
in $I$. For each $n$, fix $\pi_n\in \psi(y_n)$. Since
$\psi(y_n)\to \psi(x)$ we have
\[
\lim_{n \to \infty} \tr\big( \pi_n(f)\big)
    = \tr \big(\pi(f)\big)
    = \tr  \Theta_{\xi_\mu, \xi_\mu}
    = 1.
\]
But $\{y_n,z_n\}\in\psi^{-1}(\{\psi(y_n)\})$, so by Lemma
\ref{lem:D-points} and the positivity of $f$ we also have
\[
    \tr\big(\pi_n(f)\big)\geq y_n(f) + z_n(f)\to 2x(f)=2,
\]
which results in a contradiction.
Thus $\psi$ is  
a local homeomorphism.

(2)  Since $\psi : \widehat{B} \to \widehat{A}$ is a local
homeomorphism by (1), the collection
\[
\Uu(\psi) := \lbrace U \subset \widehat{B} \text{ open}
                     :\psi|_U  \text{ is injective }\rbrace
\]
forms an open cover of $\widehat{B}$.  Since $B$ is separable,
$\widehat{B}$ is second-countable and hence paracompact. It
follows by \cite[Lemma 2.1]{Kumjian1985} (see also the Shrinking
Lemma \cite[Lemma 4.32]{tfb}) that there is a countable, locally
finite refinement
$\Vv := \lbrace  V_n : n \ge 0\}$ of $\Uu$ such that $V_i \cap V_j \ne \emptyset$ implies $V_i \cup V_j \in \Uu(\psi)$.

By definition of the extension property, $B$ contains an
approximate identity for $A$. Since $A$ is liminary, we may
apply Lemma~\ref{lem:D-points} to conclude that $P$ is
faithful.   By Lemma~\ref{lem:exp kills normalisers} we have
$N_f(B)\subset\ker P$, so it remains to show that every element
in $\ker(P)$ may be approximated by sums of elements in
$N_f(B)$.

By \cite[Lemma 4.34]{tfb} there exists a partition of unity
subordinate to $\Vv$; that is, there exists a sequence
$(f_n)^\infty_{n=0}$ in $B$ such that $\supp f_n \subset V_n$,
$0 \le f_n \le 1$ and $\sum_{n=0}^\infty x(f_n) = 1$ for all $x
\in \widehat{B}$. For each $n \ge 0$, let $g_n = \sum_{j=0}^n
f_j$. For a compact subset $K$ of $\widehat{D}$, the local
finiteness of $\Vv$ implies that $K \cap V_j = \emptyset$ for
all but finitely many $j$. Hence there exists $n \ge 0$ such
that $g_n(x) = 1$ for all $x \in K$. Therefore,
$(g_n)^\infty_{n=0}$ is an approximate identity for $B$ and
hence for $A$. Fix $a \in A$. Then $g_n a g_n \to a$ and
$P(g_nag_n) = g_nP(a)g_n$ for all $n$. Fix $x \in \widehat{B}$.
Since $x(f_i)x(f_j) = 0$ whenever $x \not\in V_i \cap V_j$, we
obtain
\[
x \circ P(g_nag_n) = x (g_nP(a)g_n) =
           \sum_{\{0\leq i,j\leq n:x \in V_i \cap V_j\}} x(f_i) \big(x \circ P(a)\big) x(f_j).
\]
Hence
\[
P(g_nag_n) =
           \sum_{\{0\leq i,j\leq n: V_i \cap V_j\neq\emptyset\}} f_i P(a) f_j.
\]
Suppose that $V_i \cap V_j \ne \emptyset$. Then $V_i \cup V_j
\in \Uu(\psi)$, so $\psi|_{V_i\cup V_j}$ is injective, and
Lemma~\ref{lem:normaliser fairy} gives $f_i a  f_j  \in B$.
Hence  $f_i P(a) f_j = P(f_i a f_j) = f_i a  f_j$ and we have
\begin{equation}\begin{split}\label{eq-kernel}
P(g_nag_n) &= \sum_{\{0\leq i,j\leq n: V_i \cap V_j\neq\emptyset\}} f_i  a f_j, \quad \text{and} \\
(I - P)(g_nag_n) & = \sum_{\{0\leq i,j\leq n: V_i \cap V_j=\emptyset\}} f_i  a f_j.
\end{split}
\end{equation}
Suppose now that $a \in \ker P$. Then $g_nag_n \in \ker P$.
Since $g_nag_n \to a$, it suffices to show that $g_nag_n$ may
be expressed as a sum of free normalisers. Using
\eqref{eq-kernel} and $P(g_nag_n) = 0$ we have
\[
g_nag_n = \sum_{\{0\leq i,j\leq n: V_i \cap V_j=\emptyset\}}f_i  a f_j.
\]
Since $\psi|_{V_k}$ is injective and $\supp f_k \subset V_k$  for
all $k$, Lemma \ref{lem:normaliser fairy} gives $f_i  a f_j \in N(B)$.
If $V_i \cap V_j = \emptyset$, then $f_i  f_j = 0$ so that
$(f_iaf_j)^2=0$. Thus $f_i  a f_j \in N_f(B)$ as required, and
$(A,B)$ is a diagonal pair.
\end{proof}

We now give an example of a bounded-trace $C^*$-algebra $A$ with
a maximal abelian subalgebra $B$ such that $(A,B)$ has the
extension property but is not a diagonal pair. Thus
Theorem~\ref{thm:ext prop => diag} cannot be extended from Fell
algebras to bounded-trace $C^*$-algebras.

\begin{example}\label{ex:alex}
Let $C := \{f \in C([0,1],M_2) : f(0) \in \CC I_2\}$, and let
$D$ be the subalgebra of $C$ consisting of functions $f$ such
that each $f(t)$ is a diagonal matrix. Then $C$ is a
bounded-trace algebra, but is not a Fell algebra, and $D$ is an
abelian $C^*$-subalgebra. Each pure state of $D$ has the form
$d \mapsto d(t)_{i,i}$ for some $t \in [0,1]$ and $i \in
\{1,2\}$, and then $c \mapsto c(t)_{i,i}$ is the unique
extension to a pure state of $C$, so $D$ has the extension
property relative to $C$.

For $t > 0$, let $u_t := \big(\begin{smallmatrix} \cos(1/t) &
\sin(1/t) \\ -\sin(1/t) & \cos(1/t)\end{smallmatrix}\big) \in
M_2(\CC)$. Define $\alpha \in \Aut(C)$ by
\[
\alpha(f)(t) =
    \begin{cases}
        u_t f(t) u^*_t &\text{ if $t > 0$} \\
        f(0) &\text{ if $t = 0$.}
    \end{cases}
\]

Let
\[
A := M_2(C)\qquad\text{and}\qquad B :=
    \Big\{\Big(\begin{array}{cc}d_1 & 0 \\ 0 & \alpha(d_2) \end{array}\Big) : d_1, d_2 \in D\Big\}.
\]
Then $A$ is not a Fell algebra but has bounded trace; $B$ is abelian, and $B$ has the
extension property relative to $A$ because each of $D$ and
$\alpha(D)$ has the extension property relative to $C$. The
unique faithful conditional expectation $P : A \to B$ is given
by
\[
P := \Big(\begin{array}{cc} \phi & 0 \\ 0 & \alpha \phi \alpha^{-1} \end{array}\Big),
\]
where $\phi$ is the canonical expectation from $C$ onto $D$:
$\phi(c)(t) = \big(\begin{smallmatrix} c_{1,1}(t) & 0 \\ 0 &
c_{2,2}(t)\end{smallmatrix}\big)$.

We claim that $B$ is not diagonal in $A$. First observe that if
$n$ is a normaliser of $D$ in $C$, then there exists
$\lambda(n) \in \CC$ such that $n(0) = \lambda(n)I_2$ by
definition of $C$. Hence the off-diagonal entries of $n(t)$ go
to zero as $t$ goes to zero. Since $n$ is a normaliser, for $t
> 0$ the matrix $n(t)$ is either diagonal, or a linear combination
of the off-diagonal matrix units. In particular, if $n(0) \not=
0$ then by continuity, $n(t)$ is diagonal for $t$ in some
neighbourhood of $0$. If $n$ is a \emph{free} normaliser, then
each $n(t)^2 = 0$, and it follows from the above that $n(0) =
0$.

Now suppose that
\[
n = \Big(\begin{array}{cc} n_{1,1} & n_{1,2} \\ n_{2,1} & n_{2,2} \end{array}\Big) \in A
\]
is a normaliser of $B$. We claim that $n_{1,2}(0) = 0$. Note that for $t
> 0$, each of $n_{1,1}(t)$, $u_t^* n_{2,2}(t) u_t$ is a
normaliser of $D(t)$ and each of $n_{1,2}(t) u_t$, and $u^*_t
n_{2,1}(t)$ is a normaliser of $D(t)$ in $C(t)$. Suppose for
contradiction that  $n_{1,2}(0) \not= 0$.
Since $n_{1,2}(0)$ is diagonal, there exists $\varepsilon$ such
that for $t < \varepsilon$, both $\|n_{1,2}(t)\| >
\|n_{1,2}(0)\|/\sqrt{2}$ and $|(n_{1,2}(t))_{i,j}| <
\|n_{1,2}(0)\|/2$ for $i \not= j$. Choose $t_0 < \varepsilon$
such that $u_{t_0} = \big(\begin{smallmatrix} 1/\sqrt{2} &
1/\sqrt{2} \\ -1/\sqrt{2} & 1/\sqrt{2}\end{smallmatrix}\big)$.
Since $n_{1,2}(t_0)u_{t_0}$ is a normaliser of $D(t_0)$, it is either
diagonal, or a linear combination of off-diagonal matrix units,
and since $t_0 < \varepsilon$, there is an entry of
$n_{1,2}(t_0)u_{t_0}$ with modulus at least
$\|n_{1,2}(0)\|/\sqrt{2}$. It follows by choice of $u_{t_0}$
that $n_{1,2}(t_0) = (n_{1,2}(0)u_{t_0}) u^*_{t_0}$ has at
least one off-diagonal entry of modulus greater than
$(\|n_{1,2}(0)\|/\sqrt{2})(1/\sqrt{2}) = \|n_{1,2}(0)\|/2$,
contradicting the choice of $\varepsilon$. Hence $n_{1,2}(0) = 0$ as
claimed.

The function $f : [0,1] \to M_4$ given by
\[
f(t) = \big(\begin{smallmatrix} 0_2 & I_2 \\ 0_2 & 0_2 \end{smallmatrix}\big)
\]
belongs to $\ker(P) \subset A$. But $f(0)_{1,2} \not=0$, so $f$ is
not in the closed span of the normalisers of $B$ in
$A$, and hence is not in the closed span of the free normalisers. In particular $B$ is not diagonal in $A$.
\end{example}

We will show that every separable Fell algebra is Morita
equivalent to one with a diagonal subalgebra; to do this we need:

\begin{lemma}
Let $A$ be a separable Fell algebra. Then there exists a
countable set of abelian elements of $A$ which generate $A$ as
an ideal.
\end{lemma}

\begin{proof}
By Lemma~\ref{lem-elementary?}, for every irreducible
representation $\pi$ of $A$ there exists an abelian element
$a_\pi$ of $A$ such that $\pi(a_\pi)\neq 0$.  For each $\pi$,
the set $U_\pi := \{[\sigma]\in\widehat A : \sigma(a_\pi) \ne 0
\}$
is an open neighbourhood of $[\pi]$ in
$\widehat A$. Since $A$ is separable, the topology for $\widehat A$ has
a countable base \cite[Proposition~3.3.4]{dix}. So there exists
a countable subset $S:=\{a_{\pi_i} : i \in \NN\}$ such that
$\{U_{\pi_i} : i \in \NN\}$ is an open cover of $\widehat A$. Let
$I$ be the ideal generated by $S$.  Since
$\sigma(a_{\pi_i})\neq 0$ when $[\sigma]\in U_{\pi_i}$, it
follows that  $\pi|_I\neq 0$ for every irreducible
representation $\pi$ of $A$. Hence $I=A$.
\end{proof}

\begin{thm}\label{prp:construct diagonal}
Let $A$ be a separable Fell algebra, and let $\{a_i : i \in
\NN\}\subset A$ be a set of abelian norm-one elements which
generate $A$ as an ideal. Let $\Kk = \Kk(\ell^2(\NN))$, and
denote the canonical matrix units in $\Kk$ by $\{\Theta_{i\,j} :
i,j \in \NN\}$. Set
\[
a:=\sum_{i=1}^\infty \frac{1}{i} a_i \otimes \Theta_{i\,i} \in A \otimes \Kk
\]
Then
\begin{enumerate}
\item  The hereditary subalgebra $C := \overline{a (A
    \otimes \Kk) a}$ generated by $a$
    is Morita equivalent to $A$; and 
\item $D := \bigoplus_{i \in \NN} \overline{a_i A a_i}
    \otimes \Theta_{i\,i}$ is a $C^*$-diagonal in $C$; and
\item the conditional expectation $P:C\to D$ is given by
    $P(c)=\bigoplus_i(1\otimes\Theta_{i\,i})c(1\otimes\Theta_{i\,i})$.
\end{enumerate}
\end{thm}
\begin{proof} For~(1), it suffices to show that
$C$ is full  or, equivalently, that $\overline{(A \otimes \Kk) a (A \otimes \Kk)} = A \otimes \Kk$.
Since $A$ is generated by the $a_i$, it suffices to show that for all $i, j,k \in \NN$
\[
a_i \otimes \Theta_{j\,k} \in \overline{(A \otimes \Kk) a
(A \otimes \Kk)}.
\]
Fix $i,j,k \in \NN$ and let $(e_\lambda)_{\lambda \in
\Lambda}$ be an approximate identity for $A$. Then
\[
a_i \otimes \Theta_{j\,k}
    =i \lim_{\lambda \in \Lambda} (e_\lambda \otimes \Theta_{j\,i})
                                    a (e_\lambda \otimes \Theta_{i\,k})
    \in \overline{(A \otimes \Kk) a (A \otimes \Kk)}
\]
as required.

For~(2), first observe that $D$ is commutative because each
$a_i$ is an abelian element.  Since $A$ is a Fell algebra so is
$C$. So by Theorem~\ref{thm:ext prop => diag}, to see that $D$
is diagonal in $C$, it suffices to prove that $D$ has the
extension property relative to $C$. By
Remark~\ref{rmk:nonunital}(2), it is enough to show that $D +
\lsp[D, C]$ is dense in $C$.

Sums of the form
\[
\sum_{j,k = 1}^n a_j b_{j\,k} a_k \otimes\Theta_{j\,k},
\]
with $b_{j\,k} \in A$, are dense in $C$. It therefore suffices
to show that elements of the form $a_j b_{j\,k} a_k \otimes
\Theta_{j\,k}$ with $j \ne k$ may be approximated by elements in
$[D,C]$.  Fix $c := a_j b_{j\,k} a_k \otimes \Theta_{j\,k}$ with
$j \ne k$.  For $n \in \NN$, let $d_n := a_j^{1/n} \otimes
\Theta_{j\,j} \in D$. Since $a_j^{1/n}a_j \to a_j$ as $n \to
\infty$,
\[
[d_n, c] = d_n c - cd_n  = a_j^{1/n}a_j b_{j\,k} a_k \otimes \Theta_{j\,k}
            \xrightarrow{n \to \infty} a_j b_{j\,k} a_k \otimes \Theta_{j\,k} = c.
\]
Hence $c$ may be approximated by the commutators $[d_n, c]$,
and so $D + \lsp[D, C]$ is dense in $C$.

For~(3), observe that the formula given for $P$ determines a
norm-decreasing projection of $C$ onto $D$. This is then a
conditional expectation by Remark~\ref{rem-cond-exp}, and is the
unique expectation from $C$ to $D$ as discussed in
Notation~\ref{rmk:def-psi}.
\end{proof}

\section{Fell algebras and twisted groupoid \texorpdfstring{$C^*$}{C*}-algebras}\label{sec:TypeI0 and diagonals}

In \cite[Theorem~3.1]{Kumjian1986} Kumjian showed that if
$(A,B)$ is a diagonal pair, then $A$ is isomorphic to a
twisted groupoid $C^*$-algebra. Here we combine this with the results
of Section~\ref{sec:ext props} to show that up to Morita
equivalence every Fell algebra arises as a twisted groupoid
$C^*$-algebra, and conversely determine for which twists the
associated twisted groupoid $C^*$-algebra is Fell. We start
with some background from \cite[\S2]{Kumjian1986}.

\label{defn twist} A \emph{$\TT$-groupoid} $\Gamma$ is a
locally compact, Hausdorff groupoid $\Gamma$ equipped with a
free range- and source-preserving action of the circle group
$\TT$  such that $(t_1 \cdot \gamma_1)(t_2 \cdot \gamma_2) =
(t_1t_2) \cdot (\gamma_1\gamma_2)$ whenever
$(\gamma_1,\gamma_2)$ is a composable pair in $\Gamma$. The
quotient groupoid $\Gamma/\TT$ is  Hausdorff because $\TT$ is compact.

Recall that a sequence $K^{(0)} \hookrightarrow K
\stackrel{i}{\to} G \stackrel{q}{\to} H$ of groupoids is
\emph{exact} if $q$ is a surjective groupoid homomorphism which
restricts to an isomorphism of unit spaces, and $i$ is an
isomorphism of $K$ onto $\ker(q) = \{g \in G : q(g) \in
H^{(0)}\}$. A \emph{topological twist} or just \emph{twist} is a
$\TT$-groupoid $\Gamma$ such that there is an exact
sequence\label{pg-twist2}
\[
\Gamma^{(0)}\to\Gamma^{(0)} \times \TT \to \Gamma \stackrel{q}{\to} R
\]
of groupoids in which $R$ is a principal, \'etale groupoid (a
\emph{relation} in the terminology of \cite{Kumjian1986}). Note
that $\Gamma^{(0)}=R^{(0)}$. We often abbreviate the exact
sequence to $\Gamma\to R$. Twists $\Gamma_1 \stackrel{q_1}{\to}
R$ and $\Gamma_2 \stackrel{q_2}{\to} R$ over the same relation
$R$ are \emph{isomorphic} if there is a $\TT$-equivariant
isomorphism $\pi : \Gamma_1 \to \Gamma_2$ such that $q_2 \circ
\pi = q_1$; we call $\pi$ a \emph{twist isomorphism}. A twist
$\Gamma \stackrel{q}{\to} R$ is said to be
\emph{trivial}\label{pg:trivial} if $q$ has a continuous
section which is a groupoid homomorphism.  A trivial twist over
$R$ is isomorphic to the cartesian-product groupoid $R \times
\TT$ \cite[Remark~4.2]{Kumjian1986}.\label{pg:trivial twist}

We outline in the appendix the
construction of the twisted groupoid $C^*$-algebra
$\tgcsa{\Gamma}{R}$ associated to a twist, and also prove there
that the $C^*$-algebra of a trivial twist is isomorphic to the
reduced groupoid $C^*$-algebra $C^*_{\red}(R)$ of $R$. In brief,
$\tgcsa{\Gamma}{R}$ is a $C^*$-completion of the collection of
$C_c(\Gamma;R)$ of compactly supported $\TT$-equivariant
functions on $\Gamma$; the closure of the algebra of sections in
$C_c(\Gamma;R)$ which are supported on $\TT \cdot \Gamma^0$ can
be identified with $C_0(\Gamma^{(0)})$, and restriction of
functions extends to a conditional expectation $P :
\tgcsa{\Gamma}{R} \to C_0(\Gamma^{(0)})$\label{pg:tgcsa outline}.

For our classification theorem, a key tool will be the following
theorem, proved in \cite{Kumjian1986}.

\begin{thm}\label{thm-alex}
\textup{(\cite[Theorem~3.1]{Kumjian1986})} Let $A$ be a
separable $C^*$-algebra with diagonal $B$, and let $Y :=
\widehat{B}$. Then there exists a twist $\Gamma\to R$,
a homeomorphism $\phi : Y \to \Gamma^{(0)}$, and an isomorphism
$\pi : A \to \tgcsa{\Gamma}{R}$ such that the following diagram
commutes.
\begin{equation}\label{eq:twist CD}
\begin{CD}
B @>{\phi_*}>> C_0(\Gamma^{(0)}) \\
@V{\subseteq}VV @V{\subseteq}VV \\
A @>{\pi}>> \tgcsa{\Gamma}{R}
\end{CD}
\end{equation}
\end{thm}

Since we need the details below, we now sketch the construction of the twist $\Gamma$
from a unital diagonal pair $(A,B)$ given in \cite[Theorem~3.1]{Kumjian1986}; Remark~\ref{nonunital} below explains how the construction works for nonunital diagonal pairs.  Let $Y=\widehat B$ and set
$$
\Gamma_0=\{(a,y)\in N(B)\times Y: y(a^*a)>0\}.
$$
For $y \in Y$ we continue to write $y$ for the unique
state extension to $A$, and then for each $(a,y)\in \Gamma_0$,
we define $[a,y] : A \to \CC$ by $
[a,y](c)=y(a^*c)y(a^*a)^{-1/2}$. Then each $[a,y]$ belongs to
the dual space $A^*$ of $A$, and the following are equivalent:
(1) $[a,y]=[c,y]$; (2) $y(a^*c)>0$; (3) there exist $b_1$,
$b_2\in B$ with $y(b_1), y(b_2) > 0$ such that $ab_1=cb_2$. Set
\begin{equation}\label{eq:Gamma def}
\Gamma:=\{[a,y]:(a,y)\in\Gamma_0\}\subset A^*.
\end{equation}
Define a $\TT$-action on $\Gamma$ by scalar multiplication:
$t\cdot[a,y]=\big[\,\overline{t}a,y\big]$; this agrees with scalar
multiplication on  $A^*$ but not with the convention used in
\cite[\S{5}]{Renault2008}.

By \cite[Proposition~1.6]{Kumjian1986}, for each $a\in N(B)$,
there is a homeomorphism
\[
\sigma_a:\{y\in Y: y(a^*a)>0\}\to \{y\in Y: y(aa^*)>0\}
\]
such that $y(a^*ba)=\sigma_a(y)(baa^*)$ for all $b\in B$ and
all $y$ in the domain of $\sigma_a$. The set $\Gamma$, with
source and range maps defined by $s([d,y])=y$ and
$r([d,y])=\sigma_d(y)$, and partial multiplication defined by
$[a,\sigma_c(y)][c,y]=[ac, y]$, is a $\TT$-groupoid. The
quotient groupoid $R=\Gamma/\TT$ is a principal \'etale
groupoid, and $\Gamma \to R$ is a twist satisfying the
requirements of Theorem~\ref{thm-alex}.  The class
of this twist is the negative of the one constructed in
\cite[\S{5}]{Renault2008}.

\begin{remark}\label{nonunital}
The construction outlined above is for unital diagonal pairs
$(A,B)$. However, as mentioned in the proof of
\cite[Theorem~3.1]{Kumjian1986}, the construction may be
applied to nonunital pairs as follows. When $(A,B)$ is a
nonunital diagonal pair, one applies the above construction to
the diagonal pair $(\widetilde{A}, \widetilde{B})$ to obtain a
twist $\widetilde{\Gamma} \to \widetilde{R}$ with unit space
\[
\widetilde{\Gamma}^{(0)}
    = \widetilde{R}^{(0)}
    = \widehat{B} \cup \{\infty\}.
\]
It is straightforward to see that $\widehat{B} \subset
\widetilde{\Gamma}^{(0)}$ is an open invariant subset, so we
may restrict both $\widetilde{\Gamma}$ and $\widetilde{R}$ to
$\widehat{B}$ to obtain a twist $\Gamma \to R$. It is routine
to check that $\tgcsa{\Gamma}{R}$ may be identified with an
ideal $I \lhd \tgcsa{\widetilde{\Gamma}}{ \widetilde{R}} \cong
A$ for which the quotient is isomorphic to $\tgcsa{\TT}{ \{1\}}
= \CC$. Hence $I$ coincides with $A \lhd \widetilde{A}$, and it
is clear from the construction that this identification takes
$I \cap C_0(\widetilde{\Gamma}^{(0)}) = C_0(\Gamma^{(0)})$ to
$B$. In particular, there is an isomorphism $\pi : A \to
\tgcsa{\Gamma}{R}$ which makes the diagram~\eqref{eq:twist CD}
commute.

We claim that $\Gamma$ is still described by~\eqref{eq:Gamma
def}. This is not obvious right off the bat: by definition the
elements of $\Gamma$ are of the form $[n,x]$ where $n$ is a
normaliser of $\widetilde{B}$ in $\widetilde{A}$, and $x$
belongs to $\widehat{B} \subset \widehat{\widetilde{B}}$. So we
must show that if $n = (n',\lambda) \in \widetilde{A}$
normalises $\widetilde{B}$ and $x \in s(n) \setminus
\{\infty\}$, then $[n,x] = [(m,0),x]$ for some normaliser $m$
of $B$ in $A$.

Fix $u \in \Gamma^{(0)}$. Then $u$ has the form $[b_0,x]$ where
$b_0 \in \widetilde{B}^+$ and $x \in \widehat{\widetilde{B}}$.
Moreover, if $x \not= \infty$, then there exists $b \in B^+
\subset \widetilde{B}^+$ such that $b(x) > 0$, and then $[b,x] =
[b_0,x]$. Now for any $n \in N(\widetilde{B}) \subset
\widetilde{A}$ and any $x \in
\{y\in\widehat{\widetilde{B}}:y(n^*n)>0\}$, we can express
$s([n,x]) = [b,x]$ where $b \in B \subset \widetilde{B}$, and
then $[n,x] = [n,x][b,x] = [nb, x]$. We have $nb \in A$ because
$A$ is an ideal in $\widetilde{A}$, and $nb$ also normalises
$B$: for $c \in B$,
\begin{equation}\label{eq:adjusting normalisers}
(nb)^* c (nb) = b^* (n^* c n) b\qquad\text{ and }\qquad
(nb) c (nb)^* = n (bcb^*) n^*,
\end{equation}
and both belong to $B$ because $n$ normalises $\widetilde{B}$
and $A$ is an ideal in $\widetilde{A}$.
\end{remark}

\begin{prop}\label{thm-equivalencerelationoftwist}
Let $(A,B)$ be a diagonal pair such that $A$ is a separable
Fell algebra, and let $\Gamma \to R$ be the twist constructed
from $(A,B)$ as above. Let $\psi : \widehat{B} \to \widehat{A}$
be the spectral map. Then for $x,y \in Y$, there exists $\alpha
\in R$ such that $r(\alpha) = x$ and $s(\alpha) = y$ if and
only if $\psi(x) = \psi(y)$. Furthermore, the map $\alpha
\mapsto (r(\alpha), s(\alpha))$ is a topological groupoid
isomorphism from $R$ onto $R(\psi)$.
%
\end{prop}

\begin{proof}
Let $P$ be the conditional expectation from $A$ to $B$.
Recall that for $x \in \widehat{B}$, we have $\psi(x) =
\psi_P(x) = [\pi_x]$ where  $(\pi_x, \Hh_x, \xi_x)$ is the GNS
triple associated with the pure state $x\circ P$; so we have
$x\circ P(a)=(\pi_x(a)\xi_x\,|\,\xi_x)$ for all $a\in A$ (see
Notation~\ref{rmk:def-psi}).

First fix $\alpha \in R$, and let $x = r(\alpha)$ and $y =
s(\alpha)$. By definition of $R = \Gamma/\TT$, there exists
$n\in N(B)$ such that $y(n^*n) > 0$ and $x = \sigma_n(y)$.  By
scaling $n$ we may assume that $y(n^*n) = 1$. Since $n^*n \in B$
it follows that for $b \in B$, we have $y(b) = y(b)y(n^*n) = y(b
n^*n)$. So by definition of $\sigma_n$, we have $x(b) =
y(n^*bn)$ for all $b \in B$. Since $y(n^*n) = 1$, the vector
$\eta_y := \pi_y(n)\xi_y$ has norm 1. Now
$x(b)=y(n^*bn)=(\pi_y(b)\eta_y\, |\, \eta_y)$ for all $b\in B$.
Since $B$ has the extension property relative to $A$, $x\circ P$
and $a\mapsto (\pi_y(a)\eta_y\, |\, \eta_y)$ coincide on $A$.
Hence, $\pi_y$ and $\pi_x$ are unitarily equivalent, whence
$\psi(x) = \psi(y)$.

Conversely, suppose $\psi(x) = \psi(y)$.  Then the GNS
representations 
$\pi_x$ and $\pi_y$  are unitarily equivalent, so there is an
irreducible representation $\pi : A \to B(\Hh)$ and unit vectors
$\xi, \eta \in \Hh$ such that $y(P(\cdot)) = ( \pi(\cdot)\xi
\,|\, \xi )$ and $x(P(\cdot)) = ( \pi(\cdot)\eta \,|\, \eta )$.
Since $\pi(A) = \Kk(\Hh)$, there exists $a \in A$ such that
$\pi(a) = \Theta_{\eta, \xi}$. Since $A$ is a Fell algebra,
Theorem~\ref{thm:ext prop => diag}(1) implies that $\psi$ is a
local homeomorphism, so there exist open neighbourhoods $U$ of
$y$ and $V$ of $x$ such that $\psi|_U$ and $\psi|_V$ are
injective. Fix and norm-one positive functions $f,g$ with compact
support such that $\supp(f) \subset U$, $\supp(g) \subset V$,
and $f(y) = g(x) = 1$. Then $\pi(f)\xi=\xi$ and
$\pi(g)\eta=\eta$. Let $n := gaf$. Then
$y(n^*n)=(\pi(gaf)\xi\,|\,\pi(gaf)\xi)=(\eta\,|\,\eta)=1$ and
$y(n^*bn)=(\pi(b)\pi(gaf)\xi\,|\,\pi(gaf)\xi)=x(b)$ for all
$b\in B$. Lemma~\ref{lem:normaliser fairy} implies that $n \in
N(B)$, so $\alpha := q([n,y]) \in R$ with $r(\alpha) = x$ and
$s(\alpha) = y$.

It remains to prove that the map $\Upsilon : \alpha \mapsto
(r(\alpha), s(\alpha))$ is a homeomorphism. It follows from the
above that $\Upsilon$ is surjective, and it is injective since
$R$ is principal. It is continuous because the range and source
maps are continuous from $R$ to $Y$. We must now show that
$\Upsilon$ is open. For this, fix $\alpha = q([a_0,x]) \in R$.
Fix neighbourhoods $U_0$ of $\sigma_{a_0}(x)$ and $V_0$ of $x$
in $\widehat{B}$ such that $\psi|_{U_0}$ and $\psi|_{V_0}$ are
homeomorphisms, and fix $b, c \in B^+$ with $\supp b \subset
U_0$ and $\supp c \subset V_0$ such that $b(\sigma_{a_0}(x)) =
c(x) = 1$. As in~\eqref{eq:adjusting normalisers}, the element
$a := b a_0 c$ is a normaliser of $B$, and
\[
U := \{y \in \widehat{B} : aa^*(y) > 0\} \subset U_0\quad\text{ and }
V := \{y \in \widehat{B} : a^*a(y) > 0\} \subset V_0.
\]
Hence $[a,x] = [a_0,x]$, so $W := \{q([a,y]) : y \in V\}$ is an
open neighbourhood of $[a,x]$ (see
\cite[page~982]{Kumjian1986}). Since $\psi(\sigma_a(y)) =
\psi(y)$ for all $y$, and since $\psi$ is injective on $U$ and
$V$, we have $\Upsilon(W) = U *_\psi V = (U \times V) \cap
R(\psi)$, and hence $\Upsilon(W)$ is open in the relative topology. So each
$\alpha \in R$ has a neighbourhood $W$ such that $\Upsilon(W)$
is open, and it follows that $\Upsilon$ is open.
\end{proof}

The following is a rewording of
\cite[Definition~5.5]{Kumjian1986}.
\begin{defn}\label{dfn:twists equivalent}
Twists $\Gamma_i \to R_i\ (i=1,2)$ are \emph{equivalent} if
there exist a twist $\Gamma \to R$ and maps $\iota_i : R_i \to
R$ such that
\begin{enumerate}
\item each $U_i := \iota_i(R^{(0)}_i)$ is a full (see page~\ref{def-full}) and open
    subset of $R^{(0)}$,
\item $R^{(0)} = U_1 \sqcup U_2$, and
\item each $\iota_i$ is an isomorphism onto $U_i R U_i$ and
    the pullback $\iota_i^*(\Gamma)$ is isomorphic to
    $\Gamma_i$.
\end{enumerate}
We call  $\Gamma\to R$ a \emph{linking twist}.
\end{defn}

The following lemma will be used in the proof of
Theorem~\ref{thm-main}  below.

\begin{lemma}\label{lem-twists}
Let $(C_1, D_1)$ and $(C_2, D_2)$ be diagonal pairs and suppose
that $C_1$ and $C_2$ are separable Fell algebras. Then $C_1$
and $C_2$ are Morita equivalent if and only if the associated
twists obtained from Theorem~\ref{thm-alex} are equivalent.
\end{lemma}
\begin{proof}
For the ``only if" implication, let $X$ be a
$C_1$--$C_2$-imprimitivity bimodule, and let $L$ be the
associated linking algebra. Let $q_1, q_2 \in M(L)$ be the
multiplier projections such that $q_i L q_i \cong C_i$ and $q_1
L q_2 \cong X$, and identify the $C_i$ and $X$ with subsets of
$L$ under these isomorphisms. By
\cite[Proposition~5.4]{Kumjian1986}, it suffices to show that $D
:= D_1 \oplus D_2$ is a diagonal in $L$.  Since $L$ is a Fell
algebra, by Theorem~\ref{thm:ext prop => diag} it suffices to
show that $D$ has the extension property relative to $L$. Let
$x$ be a pure state of $D$. Since $\widehat{D} = \widehat{D}_1
\sqcup \widehat{D}_2$, $x$ is a pure state of $D_i$ for either
$i=1$ or $i =2$ (but not both) and thus extends uniquely to a
pure state of $C_i = q_i L q_i$ because $(C_i, D_i)$ is a
diagonal pair. Since all states extend uniquely from hereditary
subalgebras \cite[Proposition 3.1.6]{Ped}, $x$ has a unique
extension to $L$, so $D$ has the extension property relative to
$L$ as required.

The ``if'' implication is \cite[Proposition~5.4]{Kumjian1986}.
\end{proof}


\begin{thm}\label{thm-main}
\begin{enumerate}
\item  Suppose that $A$ is a separable Fell algebra.  Then
    there exists a locally compact, Hausdorff space $Y$, a
    local homeomorphism $\psi:Y\to\widehat A$ and a
    $\TT$-groupoid $\Gamma$ such that
    \[
    Y\to Y\times\TT\to\Gamma\to R(\psi)
    \]
is a twist, and the twisted groupoid  $C^*$-algebra
\tgcsa{\Gamma}{R(\psi)}  is Morita equivalent to  $A$.
Moreover, any two such twists are equivalent.
\item Let $Y$ be a locally compact, Hausdorff space, $X$ a
    locally compact, locally Hausdorff space, and $\psi:Y\to
    X$ a local homeomorphism. Let $Y\to Y \times \TT \to
    \Gamma \to R(\psi)$ be a twist such that $\Gamma$ is
    second-countable. Then $A := \tgcsa{\Gamma}{R(\psi)}$ is
    a Fell algebra. Let $B := C_0(Y)$, and identify $B$ with
    a subalgebra of $A$ with conditional expectation $P : A
    \to B$ as on page~\pageref{pg:tgcsa outline}. Then there
    is a homeomorphism $h : X \to \widehat{A}$ such that $h
    \circ \psi =  \psi_P$.
\end{enumerate}
\end{thm}
\begin{proof}
(1) Let $A$ be a Fell algebra.  By Theorem~\ref{prp:construct
diagonal} there exists a diagonal pair $(C,D)$ such that $C$ is
Morita equivalent to $A$. Let $Y = \widehat{D}$. By
Theorem~\ref{thm-alex} there is a twist $Y\to
Y\times\TT\to\Gamma\to R$  such that $C$ is isomorphic to
\tgcsa{\Gamma}{R} via an isomorphism which carries $D$ to
$C_0(\Gamma^{(0)})$. By
Proposition~\ref{thm-equivalencerelationoftwist}, $R \cong
R(\psi)$, where $\psi:Y\to \widehat C$ is the spectral map.
Hence $A$ is Morita equivalent to \tgcsa{\Gamma}{R(\psi)}.

Now suppose that $Y'\to Y'\times\TT\to\Gamma'\to R(\psi')$ is
another twist such that $A$ is Morita equivalent to
$\tgcsa{\Gamma'}{R(\psi')})$. Let $(C_1, D_1) =
(\tgcsa{\Gamma}{R(\psi)}), C_0(Y))$ and $(C_2, D_2) =
(\tgcsa{\Gamma'}{R(\psi')}), C_0(Y'))$. Then each $C_i$ is
Morita equivalent to $A$. So Lemma~\ref{lem-twists} implies that
the twists are equivalent.

(2)  The pair $(A, B)$  satisfies (D2)~and~(D3) by Theorem~2.9
of~\cite{Kumjian1986} and that it satisfies~(D1) is shown in the
appendix, so $(A,B)$ is a diagonal pair.  We will show that for
each $y \in Y$ there exists $f_y \in C_0(Y)$ such that $f_y$ is
abelian in $A$ and $y(f_y) > 0$,  and then use
Theorem~\ref{thm-fell} to see that $A$ is a Fell algebra.

Fix $y\in Y$. There exists a neighbourhood $U$ of $y$ in $Y$
such that $\psi|_U$ is injective. Let $f_y\in C_c(Y)$ be a
positive element of $A$ with support contained in $U$ such that
$f_y(y)\neq 0$. To see that $f_ygf_y \in C_c(Y)$ for any $g$
in the dense subalgebra  $C_c(\Gamma;R)$ of $A$, we identify
$f_y$, $g$ with sections of the complex line bundle $L$ over
$R$ as outlined in the Appendix. Note that ${f_y}$ has support
in $R^{(0)}$. A straightforward calculation yields
\[
    {f_ygf_y}(\rho)={f_y}(r(\rho))g(\rho){f_y}(s(\rho))
\]
for $\rho\in R(\psi)$. Now let $\rho=(y_1,y_2)\in \supp
f_ygf_y$. Then $\psi(y_1)=\psi(y_2)$ and
$y_1,y_2\in\supp f_y\subset U$ gives $y_1=y_2$, so $\rho$
is a unit.  Thus $f_ygf_y \in C_c(Y)$.  In particular, the
hereditary subalgebra $\overline{f_yAf_y}$ is contained in
$C_0(Y)$, hence is abelian.  Thus $f_y$ is abelian in $A$ and
$y(f_y) > 0$ as claimed.

For each $y \in Y$, set $I_y := \overline{A f_y A}$. Then each
$I_y$ is Morita equivalent to the abelian algebra
$\overline{f_y A f_y}$ by Lemma~\ref{lem-morita}. Let $J$ be
the ideal of $A$ generated by the $I_y$ (this ideal is also the
$C^*$-subalgebra of $A$ generated by the $I_y)$, and let $I = J
\cap C_0(Y)$, so $I$ is an ideal of $C_0(Y)$. Then $I$ is the
set of functions vanishing on some closed subset $K_I$ of $Y$.
But for each $y \in Y$, we have $f_y \in I$ and $y(f_y)\neq 0$.
Hence $K_I = \emptyset$, that is $I = C_0(Y)$.  In particular,
$C_0(Y)\subset J$. Since $C_0(Y)$ is diagonal in $A$, it
contains an approximate identity for $A$, so $A$ is generated
as a $C^*$-algebra by the $I_y$. Theorem~\ref{thm-fell} now
implies that $A$ is a Fell algebra.

It remains to prove that there is a homeomorphism $h : X \to
\widehat{A}$ such that $h \circ \psi =  \psi_P$.  We have
$R(\psi) \cong R(\psi_P)$ by
Proposition~\ref{thm-equivalencerelationoftwist}. Given $y,y'\in
Y$, we have
\[
  y,y'\in \psi^{-1}(\{x\}) \iff
 (y,y')\in R(\psi)=R(\psi_P) \iff
 \psi_P(y)=\psi_P(y').
\]
It follows that the assignment
\[
x\mapsto \psi_P(y)\quad\text{where $y\in\psi^{-1}(\{x\})$}
\]
gives a well-defined injective function $h:X\to\widehat A$ such
that $h\circ\psi=\psi_P$. Since $A$ is liminal $\psi_P$ is
surjective by Lemma~\ref{lem:D-points}, so $h\circ\psi=\psi_P$
implies $h$ is surjective.  Moreover, that $h\circ\psi=\psi_P$
and that $\psi$, $\psi_P$ are local homeomorphisms implies that
$h$ is continuous and open. Thus $h$ is a homeomorphism.
\end{proof}

\section{A Dixmier-Douady theorem for Fell algebras}\label{sec:DD}

Recall that a \emph{sheaf} of abelian groups over a topological
space $X$ is a pair $(B,\pi)$ where $B$ is a topological space
and $\pi:B\to X$ is a local homeomorphism such that for each
$x\in X$ the fibre $B_x:=\pi^{-1}(\{x\})$ is an abelian group.
Of particular importance are the constant sheaf
$\underline\ZZ_X$ over $X$ whose every fibre is $\ZZ$, and the
sheaf $\Tgerms_X$ of germs of continuous $\TT$-valued functions
on $X$ (see Notation~\ref{ntn:Tgerms} for details). When the
base-space $X$ is clear from context, we will often suppress
the subscript, and denote these $\underline\ZZ$ and $\Tgerms$
respectively.

Our strategy for defining an analogue of the Dixmier-Douady
invariant for a Fell algebra $A$ is as follows. We first choose
a twist $\Gamma \to R$ whose $C^*$-algebra is Morita equivalent
to $A$. The results of \cite{Kumjian1988} show that $\Gamma$
determines an element of a \emph{twist group} associated to $R$
and that this in turn determines an element of the second
equivariant-sheaf cohomology group $H^2(R, \Tgerms)$. We show
that $H^2(R, \Tgerms) \cong H^2(\widehat{A}, \Tgerms)$ to
obtain an element $\delta(A)$ of $H^2(\widehat{A}, \Tgerms)$
which we regard as an analogue of the Dixmier-Douady invariant
for $A$. The bulk of the work in the section goes towards
proving that this assignment does not depend on our choice of
twist $\Gamma \to R$.

We recall \cite[Definition~0.6]{Kumjian1988}. Let $G$ be an
\'etale groupoid and $B$ a sheaf over $G^{(0)}$.  An
\emph{action} of $G$ on $B$  is a continuous map
$\alpha :G*B:=\{(\gamma,b):\gamma\in G,b\in B_{s(\gamma)}\}\to B$,
$(\gamma,b) \mapsto \alpha_\gamma(b)$ such
that each $\alpha_\gamma : B(s(\gamma))\to B(r(\gamma))$ is an
isomorphism of  abelian groups and
$\alpha_{\gamma_1\gamma_2}=\alpha_{\gamma_1}\circ\alpha_{\gamma_2}$
when $(\gamma_1,\gamma_2)\in G^{(2)}$. It is common practice to
suppress the $\alpha$ and write $\gamma b$ for
$\alpha_\gamma(b)$, and we shall do so henceforth.  A sheaf $B$
over $G^{(0)}$ with such an action is called a
\emph{$G$-sheaf}. A \emph{$G$-sheaf morphism} $f:B_1\to B_2$ is
a sheaf morphism such that $f(\gamma b)=\gamma f(b)$ for
$\gamma\in G$ and $b\in B$. We will frequently regard the
sheaves $\underline\ZZ_{G^{(0)}}$ and $\Tgerms_{G^{(0)}}$ as
$G$-sheaves with trivial action.

Fix a topological groupoid $G$, a locally compact, Hausdorff
space $Y$, and a continuous open surjection $\psi : Y \to
G^{(0)}$. As in \cite[\S0.5]{Kumjian1988}, we may construct a
groupoid $G^\psi$ with unit space $(G^\psi) ^{(0)} = Y$ as
follows:
\[
G^\psi := \{(x,g,y) : x,y \in Y, g \in G, \psi(x)= r(g)\text{ and } \psi(y) = s(g)\},
\]
with structure maps
\begin{gather*}
r(x,g,y) = x,\quad s(x,g,y) = y,\quad (x,g,y)^{-1} =
(y,g^{-1},x),\quad\text{ and} \\
(x,g,y)(y,h,z) = (x,gh,z),
\end{gather*}
and with the relative topology inherited from the product
topology on $Y \times G \times Y$. We identify $Y$ with the
unit space $(G^\psi) ^{(0)}$  via the map $x \mapsto (x,
\psi(x), x)$. There is then a groupoid homomorphism $\pi_\psi :
G^\psi \to G$ given by
\begin{equation}\label{eq:pi_psi def}
\pi_\psi(x,g,y) = g.
\end{equation}
For the next result, recall the definition of a groupoid
equivalence from Definition~\ref{pg-equivalence}.

\begin{lemma}\label{lem:pullback groupoids equiv}
Let $G_1$ and $G_2$ be second-countable, locally compact,
Hausdorff groupoids, and let $(Z,\rho,\sigma)$ be an
equivalence from $G_1$ to $G_2$. Then for each $(x,g,y) \in
G_1^\rho$, there exists a unique element $\omega(x,g,y) \in
G_2$ such that
\begin{equation*}\label{eq:theta def}
x \cdot \omega(x,g,y) = g \cdot y.
\end{equation*}
Moreover, the map $\omega$ is a homomorphism from $G_1^\rho$ to
$G_2$, and $\Omega_{\rho,\sigma} : (x,g,y) \mapsto (x,
\omega(x,g,y), y)$ is an isomorphism from $G_1^\rho$ to
$G_2^\sigma$.
\end{lemma}
\begin{proof}
Fix $(x,g,y) \in G_1^\rho$. Then $\rho(x) = r(g) = \rho(g\cdot
y)$. Since $\rho$ induces to a bijection from $Z/G_2$ to
$G_1^{(0)}$, it follows that $x$ and $g\cdot y$ belong to the
same $G_2$-coset. Since $Z$ is a principal $G_2$-space, there
exists a unique element $\omega(x,g,y) \in G_2$ such that
$\sigma(x) = r(\omega(x,g,y))$ and $x \cdot \omega(x,g,y) = g
\cdot y$.

Since $\sigma$ is $G_1$-invariant, we have $\sigma(y) =
\sigma(g \cdot y) = \sigma(x \cdot \omega(x,g,y))$. In
particular, $\sigma(y) = s(\omega(x,g,y))$, and hence $(x,
\omega(x,g,y), y) \in G_2^\sigma$. An argument symmetric to
that of the preceding paragraph shows that $g$ is uniquely
determined by $\omega(x,g,y)$ and the formula $x \cdot
\omega(x,g,y) = g \cdot y$. Hence $\Omega_{\rho,\sigma}$ is a
bijection.

To see that $\omega$ is a homomorphism, we first check that it
maps units to units and that it intertwines the range and source
maps. This will imply that $\omega$ maps composable pairs to
composable pairs.  Let $(x, \rho(x), x) \in (G_1^\rho)^{(0)}$.
Since $x \cdot \sigma(x) = \rho(x) \cdot x$ we have $\omega(x,
\rho(x), x)  = \sigma(x)$; so $\omega$ preserves units. For
$(x,g,y) \in G_1^\rho$, we see as above that $r(\omega(x,g,y)) =
\sigma(x)$ and $s(\omega(x,g,y)) = \sigma(y)$. Thus, $\omega$
maps composable pairs to composable pairs.  Now let  $(x,g,y),
(y,h,z) \in G_1^\rho$ be a composable pair; then
\[
x \cdot \omega(x,g,y)\omega(y,h,z) = g \cdot y \cdot \omega(y,h,z) = g\cdot(h \cdot z) = (gh)\cdot z,
\]
so the uniqueness assertion of the first paragraph implies that
\begin{equation*}\label{eq:mpctn of omegas}
\omega(x,g,y)\omega(y,h,z) = \omega(x,gh,z) = \omega\big((x,g,y)(y,h,z)\big).
\end{equation*}
Hence,  $\omega$ is a homomorphism.

It is immediate that $\Omega_{\rho,\sigma}$ preserves composable
pairs. So to see that $\Omega_{\rho,\sigma}$ is also a
homomorphism, we calculate
\begin{align*}
\Omega_{\rho,\sigma}(x,g,y)\Omega_{\rho,\sigma}(y,h,z)
    &= (x, \omega(x,g,y), y)(y, \omega(y,h,z), z) \\
    &= (x, \omega(x,g,y)\omega(y,h,z), z)
    = \Omega_{\rho,\sigma}(x, gh, z).
\end{align*}
The map $\Omega_{\rho,\sigma}$ is continuous because the
structure maps on the groupoid equivalence $Z$ are continuous.
Reversing the r\^oles of $G_1$ and $G_2$ in the above yields a
continuous inverse $\Omega_{\rho,\sigma}^{-1} =
\Omega_{\sigma,\rho}$, so $\Omega_{\rho,\sigma}$ is a
homeomorphism.
\end{proof}

Let $\psi : Y \to G^{(0)}$ be a local homeomorphism as before,
and let $\pi_\psi^*$ be the pullback functor from the category
$\Sheaf{G}$ of $G$-sheaves to the category $\Sheaf{G^\psi}$ of
$G^\psi$-sheaves. So
\[
    \pi_\psi^*(B) = \{(y,b) : y \in Y, b \in B_{\psi(y)}\},
\]
and for a morphism $f : B_1 \to B_2$ of $G$-sheaves,
$\pi_\psi^*(f)(y,b) = (y,f(b))$. Let $R(\psi)$ be the
equivalence relation on $Y$ induced by $\psi$. We may regard
$R(\psi)$ as a subgroupoid of $G^\psi$ by identifying it with
$\{(x, \unit{\psi(x)}, y) : (x,y) \in R(\psi)\}$. Hence, for a
$G^\psi$-sheaf $B$ the action of $G^\psi$ on $B$ restricts to
an action of $R(\psi)$ on $B$.

By\label{pg:pi_psi,F^psi} \cite[Theorem~0.9]{Kumjian1988},
$\pi_\psi^*$ is a category equivalence between $\Sheaf{G}$ and
$\Sheaf{G^\psi}$. Indeed, the proof of
\cite[Theorem~0.9]{Kumjian1988} shows that the ``inverse"
functor $F^\psi$ is defined as follows. For a $G^\psi$-sheaf
$B$, $F^\psi(B)$ is the quotient sheaf $B/R(\psi) \in
\Sheaf{G}$. Since morphisms between $G$-sheaves are equivariant
maps, each morphism $f$ of $G^\psi$-sheaves descends to a
morphism $F^\psi(f)$ of $G$-sheaves. Specifically, $[(y,b)]
\mapsto b$ is a natural isomorphism from
$F^\psi\circ\pi_{\psi}^*$ to $\id_{\Sheaf{G}}$, and $(y, [c])
\mapsto c$ is a natural isomorphism from $\pi_\psi^* \circ
F^\psi$ to $\id_{\Sheaf{G^\psi}}$. Moreover
$\pi_\psi^*(\underline{\ZZ}_{G^{(0)}})$ is isomorphic to
$\underline{\ZZ}_Y$.

\begin{lemma}\label{lem:iUZ=Z}
Let $G$ be a groupoid, and let $U$ be a full open subset of
$G^{(0)}$, and let $\iota_U:UGU\to G$ be the inclusion map. The
functor $\iota_U^* : \Sheaf{G} \to \Sheaf{UGU}$ is an
equivalence of categories such that the $UGU$-sheaves
$\iota_U^*(\underline{\ZZ}_{G^{(0)}})$  and $\underline{\ZZ}_U$
are isomorphic.
\end{lemma}
\begin{proof}
It is straightforward to verify that $GU$ is a $G$--$UGU$
equivalence under the structure maps $\rho:=r|_{GU}$ and $\sigma:=s|_{GU}$ inherited  from $G$. Hence
Lemma~\ref{lem:pullback groupoids equiv} provides an isomorphism
$\Omega_{\sigma,\rho}$ from $(UGU)^\sigma$ to $G^\rho$, and
hence an equivalence of categories $\Omega^*_{\sigma,\rho}$ from
$\Sheaf{G^\rho\big}$ to $\Sheaf{(UGU)^\sigma}$. Composing with
the category equivalences $\pi^*_\rho$ and $F^\sigma$ discussed
above, we obtain an equivalence of categories $F^\sigma
\Omega_{\sigma,\rho}^* \pi^*_\rho : \Sheaf{G} \to \Sheaf{UGU}$.
We show that $\iota^*_U$ is naturally isomorphic to $F^\sigma
\Omega_{\sigma,\rho}^* \pi^*_\rho$. It then follows that
$\iota^*_U$ is also an equivalence of categories.

Fix $B \in \Sheaf{G}$. Then $ \iota_U^*(B)
    = \{(u,b) : u \in U, b \in B_u\}
$ and
\[
F^\sigma \Omega_{\sigma,\rho}^* \pi^*_\rho(B)
    = F^\sigma (\{(x,b) : x \in GU, b \in B_{\rho(x)}\})
    = \{[x,b] : x \in GU, b \in B_{\rho(x)}\}.
\]
The map $[g] \mapsto \sigma(g)$ from $GU/R(\sigma) \to U$ is a
bijection. It follows that
\[
F^\sigma \Omega_{\sigma,\rho}^* \pi^*_\rho(B)
    = \{[u,b] : u \in U, b \in B_u\},
\]
and that $t_B : [u,b] \mapsto (u,b)$ is an isomorphism from
$F^\sigma \Omega_{\sigma,\rho}^* \pi^*_\rho(B)$ to
$\iota_U^*(B)$. It is routine to see that for a morphism $f$ of
$G$-sheaves, $F^\sigma \Omega_{\sigma,\rho}^*
\pi^*_\rho(f)[u,b] = [u, f(b)]$, and $\iota_U^*(f)(u,b) = (u,
f(b))$, so the family of maps $t_B$ constitute a natural
isomorphism from $F^\sigma \Omega_{\sigma,\rho}^* \pi^*_\rho$
to $\iota_U^*$.

It remains to check that $\iota^*_U(\underline\ZZ_{G^{(0)}})
\cong \underline\ZZ_U$. We have
\begin{align*}
\iota_U^*(\underline{\ZZ}_{G^{(0)}})
    &= \{(x,n,y):x\in U,(n,y)\in\ZZ\times G^{(0)}, \iota_U(x)=y\} \\
    &= \{(x,n,x):x\in U,n\in\ZZ\},
\end{align*}
and the latter is isomorphic to  $\underline{\ZZ}_U$ via
$(x,n,x) \mapsto (n,x)$.
\end{proof}

For the next lemma, we need some notation.

\begin{notation+h}\label{ntn:Tgerms}
Given a topological space $X$, continuous $\TT$-valued
functions $f,g$ defined on open subsets of $X$, and a point $x
\in X$, we write $f \sim_x g$ if there exists an open
neighbourhood $W$ of $x$ with $W \subset \dom(f) \cap \dom(g)$
such that $f|_W = g|_W$. We denote by $[f]^X_x$ the equivalence
class of $f$ under $\sim_x$; this is called the \emph{germ} of
$f$ at $x$. The sheaf $\Tgerms_X$ has fibres
\[
\Tgerms_x := \{[f]^X_x : f \in C(U, \TT)\text{ for some open neighbourhood }U\text{ of }x\},
\]
with group operation $[f]^X_x + [g]^X_x := \big[(f|_{\dom(f)
\cap \dom(g)})(g|_{\dom(f) \cap \dom(g)})\big]^X_x$. For each
open set $U \subset X$ and function $f \in C(U, \TT)$, let
$O^X_{f,U} := \{[f]^X_x : x \in U\}$. The topology on
$\Tgerms_X$ has basis $\{O^X_{f,U} : U \subset X\text{ is
open}, f \in C(U, \TT)\}$. Fix an open subset $U$ of $X$. The
pullback sheaf $\iota^*_U(\Tgerms_X)$ is equal to $\{(u,
[f]^X_u) : u \in U, [f]^X_u \in \Tgerms_X\}$ with the relative
topology inherited from $X \times \Tgerms_X$; we regard
$\iota_U^*(\Tgerms)$ as the restriction of $\Tgerms$ to $U$.
\end{notation+h}

\begin{lemma}\label{lem:iU(T)=T}
Let $X$ and $Y$ be second-countable, locally compact spaces such
that $X$ is locally Hausdorff and $Y$ is Hausdorff. Let $\psi$
be a local homeomorphism from $Y$ onto an open subset of $X$. There is
an isomorphism $\phi : \psi^*(\Tgerms_{X}) \to \Tgerms_{Y}$
determined by $\phi\big(y, [f]^{X}_{\psi(y)}\big) =
[f\circ\psi]^{Y}_y$.

In particular, if $U$ is an open subset of a second-countable,
locally compact, Hausdorff space $X$ with inclusion map $\iota_U
: U \to X$, then there is an isomorphism $\phi :
\iota^*_U(\Tgerms_X) \to \Tgerms_U$ determined by $\phi(u,
[f]^X_u) = [f]^U_u$.
\end{lemma}
\begin{proof}
To see that the formula for $\phi$ is well-defined, suppose
that $(y, [f]^X_{\psi(y)}) = (z, [g]^X_{\psi(z)})$. Then $y =
z$, and there exists an open neighbourhood $V$ of $\psi(y)$ in
$X$ such that $f|_V = g|_V$. Let $U := \psi^{-1}(V)$. Then $U$
is an open neighbourhood of $y$, and $(f\circ\psi)|_U =
(g\circ\psi)|_U$ because $f$ and $g$ agree on $\psi(U)$. Hence
$[f\circ\psi]^{Y}_y = [g \circ \psi]^Y_z$.  It is routine to
check that $\phi$ is a sheaf morphism.

For surjectivity, fix an open subset $U \subset Y$, a function
$f \in C(U, \TT)$ and a point $y \in U$. We must show that
$[f]^Y_y$ belongs to the image of $\phi$. Choose a
subneighbourhood $V \subset U$ of $y$ such that $\psi|_V$ is a
homeomorphism, and define $g \in C(\psi(V), \TT)$ by $g := f
\circ (\psi|_V)^{-1}$. By definition,
\[
\phi(y, [g]^X_{\psi(y)})
    = [g \circ \psi]^Y_y
    = [f \circ (\psi|_V)^{-1} \circ \psi]^Y_y
    = [f|_V]^Y_y
    = [f]^Y_y.
\]

For injectivity, suppose that $\phi(y, [f]^X_{\psi(y)}) =
\phi(z, [g]^X_{\psi(z)})$. Then $y = z$, and there is an open $U
\subset X$ such that $\psi(y) \in U$ and $(f \circ \psi)|_U = (g
\circ \psi)|_U$. Since $\psi$ is a local homeomorphism,
$\psi(U)$ is an open neighbourhood of $\psi(y)$ in $X$.
Moreover, for $x \in \psi(U)$, say $x = \psi(z)$, we have $f(x)
= f \circ \psi (z) = g \circ \psi(z) = g(x)$, so $f$ and $g$
agree on $\psi(U)$. Hence $[f]^X_{\psi(y)} = [g]^X_{\psi(y)}$,
so $(y, [f]^X_{\psi(y)}) = (z, [g]^X_{\psi(z)})$, and hence
$\phi$ is injective. To see that $\phi$ is a homeomorphism,
recall that the basic open sets in $\Tgerms_X$ are those of the
form
\[
O^X_{f,U} := \{[f]^X_u : u \in U\}
\]
where $U$ ranges over open subsets of $X$ and $f$ ranges over
continuous $\TT$-valued functions on $U$. Since $\psi$ is a
local homeomorphism, the family of open sets $\{O^Y_{f,V} :
\psi|_V\text{ is a homeomorphism}\}$ is a basis for the
topology on $\Tgerms_Y$. The basic open neighbourhoods in
$\pi_\psi^*(\Tgerms_X)$ are by definition of the form
\[
W \ast O^X_{f,U}
    = (W \times O^X_{f,U}) \cap \pi_\psi^*(\Tgerms_X)
    = \bigcup_{w \in W} \{(w, [f]^X_{\psi(w)}) : \psi(w) \in W\}.
\]
where $W \subset Y$ is open and $O^X_{f,U}$ is a basic
open set in $\Tgerms_X$. We calculate:
\[
\phi(W \ast O^X_{f,U})
    = \bigcup_{w \in W}\{ [f \circ \psi]^Y_w : \psi(w) \in U\} \\
    = O^Y_{f \circ \psi, \psi^{-1}(U) \cap W}.
\]
If $V \subset Y$ and $\psi|_V$ is a homeomorphism, then for $f
\in C(V)$,
\begin{align*}
\phi^{-1}(O^Y_{f,V})
    &= \phi^{-1}(O^Y_{f \circ(\psi|_V)^{-1} \circ \psi, V})
    = \phi^{-1}(\{[f \circ (\psi|_V)^{-1} \circ \psi]^Y_y : y \in V\}) \\
    &= \{[f \circ (\psi|_V)^{-1}]^X_{\psi(y)} : y \in V\}
    = O^X_{f \circ (\psi|_V)^{-1}, \psi(V)},
\end{align*}
which is a basic open set because $\psi$ is open. Hence both
$\phi$ and $\phi^{-1}$ carry basic open sets to basic open sets,
and $\phi$ is a homeomorphism.

For the second statement, apply the first to $\iota_U : U \to
X$.
\end{proof}

Recall from \cite[page~215]{Kumjian1988} that given a groupoid
$G$ and a $G$-sheaf $B$, for each $n \in \NN$, the
$n$\textsuperscript{th} equivariant-cohomology group $H^n(G,B)$
is defined by $H^n(G,B) := \Ext^n_G(\underline{\ZZ}, B)$ (see
\cite{Haefliger76} for an alternative definition of sheaf
cohomology of \'{e}tale groupoids).

\begin{prop}\label{prp:H2 well-defd}
Suppose $G$ is a second-countable, locally compact, Hausdorff, \'etale groupoid, $B$ a $G$-sheaf,
and $U$ a full open subset of $G^{(0)}$. Then the inclusion $\iota_U : UGU \to G$ induces an
isomorphism $\iota_U^* : H^*(G,B) \to H^*(UGU, \iota_U^*(B))$, so in particular an isomorphism
$\iota_U^*:H^2(G,\Tgerms_{G^{(0)}}) \to H^2(UGU,
 \Tgerms_U)$.
\end{prop}
\begin{proof}
Note that $\iota^*_U(\underline{\ZZ}_{G^{(0)}}) = \underline{\ZZ}_U$ by Lemma~\ref{lem:iUZ=Z}.  So
the first isomorphism follows from applying \cite[Proposition~1.8]{Kumjian1988} to the groupoid
homomorphism $\iota_U : UGU \to G$.  In particular, there is an isomorphism
$\iota_U^*:H^2(G,\Tgerms_{G^{(0)}}) \to H^2(UGU, \iota_U^*( \Tgerms_{G^{(0)}}))$. Now
Lemma~\ref{lem:iU(T)=T} and naturality of $H^*$ imply that $H^2(UGU, \iota_U^*( \Tgerms_{G^{(0)}}))
\cong H^2(UGU, \Tgerms_U)$.
\end{proof}

\begin{cor}\label{cor:Br well-defd}
Let $X$ be a second-countable, locally compact, locally
Hausdorff space. For $i = 1,2$ fix a second-countable, locally
compact, Hausdorff space $Y_i$ and a local homeomorphism $\psi_i
: Y_i \to X$. Let $Y = Y_1 \sqcup Y_2$, and define $\psi : Y \to
X$ by $\psi|_{Y_i} = \psi_i$. Then for each $i$, the inclusion
map $\iota_{Y_i} : R(\psi_i) \to R(\psi)$ induces an isomorphism
$\iota_{Y_i}^* : H^2(R(\psi), \Tgerms_Y) \to H^2(R(\psi_i),
\Tgerms_{Y_i})$. In particular $\iota_{1,2} :=
\iota^*_{Y_2}\circ (\iota^*_{Y_1})^{-1}$ is an isomorphism from
$H^2(R(\psi_1), \Tgerms_{Y_1})$ to $H^2(R(\psi_2),
\Tgerms_{Y_2})$.
\end{cor}
\begin{proof}
The $Y_i$ are full in $R(\psi)^{(0)}$, and $Y_i R(\psi) Y_i =
R(\psi_i)$. The result now follows from Proposition~\ref{prp:H2
well-defd}.
\end{proof}

Let $\Gamma \to R$ be a twist,  $R'$ be a principal \'etale
groupoid, and  $\varphi : R' \to R$ be a continuous groupoid
homomorphism. Then the pullback twist $\varphi^*(\Gamma)$ is the
fibred product $R' \ast_\varphi \Gamma$ with structure maps
$r(\alpha,\gamma) = r(\alpha)$ and $s(\alpha,\gamma) =
s(\alpha)$, and with coordinatewise operations; it is regarded
as a twist over $R'$ under the surjection $(\alpha,\gamma)
\mapsto \alpha$.

Recall from \cite[Remark~2.9]{Kumjian1988} that given a twist
$\Gamma \stackrel{q}{\to} R$ there is an extension
\[
\Tgerms_{R^{(0)}} \to \underline{\Gamma} \to R
\]
such that $\underline{\Gamma}$ is the groupoid consisting of
germs of continuous local sections of the surjection $\Gamma \to
R$. Such extensions are called \emph{sheaf twists}, and the
group of isomorphism classes of sheaf twists over $R$ is denoted
$T_R(\Tgerms)$ (see \cite[Definition~2.5]{Kumjian1988}).
Pullbacks of sheaf twists are defined in a manner analogous to
that of the preceding paragraph. By the discussion in
\cite[Section~2.9]{Kumjian1988}, the assignment $\Gamma \mapsto
\underline{\Gamma}$ determines an isomorphism $\theta_R :
[\Gamma] \mapsto [\underline{\Gamma}]$ from the group
$\mathrm{Tw}(R)$ of isomorphism classes of twists over $R$ to
$T_R(\Tgerms)$. Moreover, suppose that $R$ is a principal
\'etale groupoid, $\Gamma$ is a twist over $R$, and $U$ is a
full open subset of $X = R^{(0)}$. Then an argument nearly
identical to that of Lemma~\ref{lem:iU(T)=T} shows that
$[\phi]^U_u \mapsto (u, [\phi]^X_u)$ determines an isomorphism
$\underline{\iota^*_U(\Gamma)} \cong
\iota^*_U(\underline{\Gamma})$. Hence, using
Lemma~\ref{lem:iU(T)=T} to identify $\iota^*_U(\Tgerms_X)$ with
$\Tgerms_U$, we see that the diagram
\begin{equation}\label{eq:pullback CD}
\begin{CD}
     \mathrm{Tw}(R)      @>>\iota_U^*>   \mathrm{Tw}(URU) \\
       @VV\theta_{R}V                   @VV\theta_{URU}V \\
T_{R}(\Tgerms_X) @>>\iota_U^*> T_{URU}(\Tgerms_U)
\end{CD}
\end{equation}
commutes.

The long exact sequence of \cite[Theorem~3.7]{Kumjian1988}
yields a boundary map $\partial^1$ from the first derived
functor $Z^1_R$ of the cocycle functor to $H^2(R, \Tgerms)$. By
\cite[Corollary~3.4]{Kumjian1988}, the twist group
$T_R(\Tgerms)$ is naturally isomorphic to $Z^1_R$, so each
twist $\Gamma$ over $R$ determines an element
$\partial^1\big(\!\big[\underline{\Gamma}\big]\!\big) \in
H^2(R, \Tgerms)$.

\begin{thm}\label{thm:DD class well-defd}
Fix a separable Fell algebra $A$. For each of $i = 1,2$ suppose
that $(C_i, D_i)$ is a diagonal pair, and that $H_i$ is an
$A$--$C_i$-imprimitivity bimodule with Rieffel homeomorphism
$h_i : \widehat{C}_i \to \widehat{A}$. For each $i$, let
$\psi_i : \widehat{D}_i \to \widehat{C}_i$ be the spectral map,
and let $\Gamma_i$ be a twist associated to $(C_i, D_i)$ as in
Theorem~\ref{thm-alex}. For each $i$, let $\tilde\psi_i := h_i
\circ \psi_i : \widehat{D}_i \to \widehat{A}$. Then the
isomorphism $\iota_{1,2} : H^2(R(\tilde{\psi}_1),
\Tgerms_{\widehat{D}_1}) \to H^2(R(\tilde{\psi}_2),
\Tgerms_{\widehat{D}_2})$ of Corollary~\ref{cor:Br well-defd}
carries
$\partial^1\big(\!\big[\underline{\Gamma}_1\big]\!\big)$ to
$\partial^1\big(\!\big[\underline{\Gamma}_2\big]\!\big)$.
\end{thm}
\begin{proof}
Since each $C_i$ is Morita equivalent to $A$, each $C_i$ is a
separable Fell algebra, and Lemma~\ref{lem-twists} implies that
$\Gamma_1$ and $\Gamma_2$ are equivalent twists. Let $\Gamma \to
R$ be a linking twist (see Definition~\ref{dfn:twists
equivalent}). Then in particular, each $\Gamma_i \cong
\widehat{D}_i \Gamma \widehat{D}_i \cong
i^*_{\widehat{D_i}}(\Gamma)$. Let
$Y:=\widehat{D}_1\sqcup\widehat{D}_2$ and define $\psi : Y \to
\widehat{A}$ by $\psi|_{\widehat{D}_i} = \tilde{\psi}_i$. Since
$\iota_{1,2} = \iota^*_{\widehat{D}_2} \circ
\big(\iota^*_{\widehat{D}_1}\big)^{-1}$ by definition, it
suffices to show that, for each of $i = 1,2$, the isomorphism
\[
\iota^*_{\widehat{D}_i} : H^2(R(\psi), \Tgerms_Y) \to
H^2(R(\tilde\psi_i), \Tgerms_{\widehat{D}_i})
\]
obtained from the first statement of Corollary~\ref{cor:Br well-defd} carries
$\partial^1\big(\!\big[\underline{\Gamma}\big]\!\big)$ to
$\partial^1\big(\!\big[\underline{\Gamma}_i\big]\!\big)$.

The naturality of the long exact sequence of
\cite[Theorem~3.7]{Kumjian1988} together with
\cite[Corollary~3.4]{Kumjian1988} implies that the right-hand
square of the diagram
\[\begin{CD}
\mathrm{Tw}(R(\tilde{\psi}_i), \TT) @>>> T_{R(\tilde{\psi}_i)}(\Tgerms_{\widehat{D}_i}) @>>{\partial^1}> H^2(R(\tilde{\psi}_i), \Tgerms_{\widehat{D}_i}) \\
@AA{\iota^*_{\widehat{D}_i}}A         @AA{\iota^*_{\widehat{D}_i}}A                         @AA{\iota^*_{\widehat{D}_i}}A \\
\mathrm{Tw}(R(\psi), \TT) @>>> T_{R(\psi)}(\Tgerms_Y) @>>{\partial^1}> H^2(R(\psi), \Tgerms_Y) \\
\end{CD}\]
commutes; the left-hand square is an instance
of~\eqref{eq:pullback CD}. Since $\Gamma$ is a linking twist
for the $\Gamma_i$, the maps $\iota^*_{\widehat{D}_i}$ on the
left of the diagram carry $\big[\underline{\Gamma}\big]$ to
$\big[\underline{\Gamma}_i\big]$. Since the diagram commutes,
it follows that the maps $\iota^*_{\widehat{D}_i}$ on the right
of the diagram carry
$\partial^1\big(\!\big[\underline{\Gamma}\big]\!\big)$ to
$\partial^1\big(\!\big[\underline{\Gamma}_i\big]\!\big)$.
\end{proof}

If $X$ and $Y$ are topological spaces, and $\psi : Y \to X$ is
a local homeomorphism, then we may regard $X$ as a groupoid
whose only elements are units, and there is then an induced
groupoid homomorphism $\pi_\psi : R(\psi) \to X$ given by
$\pi_\psi(y,z) = \psi(y)$.

\begin{prop}\label{prp:Br(X)=H2(R(psi))}
Let $X$ be a second-countable, locally compact, locally
Hausdorff space, let $Y$ be a second-countable, locally compact,
Hausdorff space, and let $\psi : Y \to X$ be a local
homeomorphism. Then $\pi_\psi^* : \Sheaf{X} \to \Sheaf{R(\psi)}$
is an equivalence of categories such that
$\pi_\psi^*(\underline{\ZZ}_X) = \underline{\ZZ}_Y$ and
$\pi_\psi^*(\Tgerms_X) \cong \Tgerms_Y$. Moreover, $\pi_\psi^*$
determines an isomorphism $\pi_\psi^* : H^*(X, \Tgerms_X) \to
H^*(R(\psi),\Tgerms_Y)$. Finally, under the hypotheses of
Theorem~\ref{thm:DD class well-defd}, $\iota_{1,2} \circ
\pi_{\tilde{\psi}_1}^* = \pi_{\tilde{\psi}_2}^*$.
\end{prop}
\begin{proof}
Regard $X$ as a groupoid with unit space $X$ whose only
morphisms are units. Then
\[
X^\psi = \{(y,x,z) : \psi(y) = x = \psi(z)\} \cong R(\psi),
\]
and under this identification the map $\pi_\psi : (y,x,z)
\mapsto x$ of~\eqref{eq:pi_psi def} agrees with the map
$\pi_\psi : R(\psi) \to X$ described above.

By \cite[Proposition~0.8 and Theorem~0.9]{Kumjian1988},
$\pi_\psi^*$ is an equivalence of categories which takes
$\underline{\ZZ}$ to $\underline{\ZZ}$. Moreover,
Lemma~\ref{lem:iU(T)=T} implies that $\pi_\psi^*$ takes
$\Tgerms$ to $\Tgerms$ also. That $\pi_\psi^*$ determines an
isomorphism of cohomologies follows from
\cite[Proposition~1.8]{Kumjian1988}.

It remains to show that $\iota_{1,2} \circ
\pi_{\tilde{\psi}_2}^* = \pi_{\tilde{\psi}_1}^*$. For this, let
$Y_i = \widehat{D}_i$ for $i = 1,2$, let $Y := Y_1 \sqcup Y_2$
and define $\tilde{\psi} : Y \to \widehat{A}$ by
$\tilde{\psi}|_{Y_i} = \tilde{\psi}_i$ as in
Corollary~\ref{cor:Br well-defd}. Consider the diagrams below.
\[
\xymatrix{
& Y\ar[dd]^{\tilde{\psi}}& & & &R(\tilde{\psi}) \ar[dd]^{\pi_{\tilde{\psi}}}& \\
Y_1 \ar[ur]^{\iota_{Y_1}}\ar[dr]_{\tilde{\psi}_1} & & Y_2\ar[ul]_{\iota_{Y_2}}\ar[dl]^{\tilde{\psi}_2}
&\qquad&
    R(\tilde{\psi}_1) \ar[ur]^{\iota_{Y_1}}\ar[dr]_{\pi_{\tilde{\psi}_1}} & &R(\tilde{\psi}_2) \ar[ul]_{\iota_{Y_2}}\ar[dl]^{\pi_{\tilde{\psi}_2}} \\
& \widehat{A} & & & & \widehat{A} &}
\]
The diagram on the left commutes by definition, and it follows
that the diagram on the right commutes also. Recall that
$\iota_{1,2} = (\iota^*_{Y_2})^{-1} \circ \iota^*_{Y_1}$ by
definition. Thus functoriality and naturality of the cohomology
exact sequence, and that $\pi_{\tilde{\psi}}^*$ takes $\Tgerms$
to $\Tgerms$ ensure that $\iota_{1,2} \circ
\pi_{\tilde{\psi}_1}^* = \pi_{\tilde{\psi}_2}^*$ as required.
\end{proof}

Theorem~\ref{thm:DD class well-defd} and
Proposition~\ref{prp:Br(X)=H2(R(psi))} ensure that we may
specify a well-defined invariant as follows.

\begin{defn}\label{dfn:DD invariant}
Let $A$ be a separable Fell algebra. Let $(C,D)$ be a diagonal
pair such that $C$ is Morita equivalent to $A$, fix an
$A$--$C$-imprimitivity bimodule, and let $h : \widehat{C} \to
\widehat{A}$ be its Rieffel homeomorphism. Let $\psi :
\widehat{D} \to \widehat{C}$ be the spectral map, and
$\tilde\psi := h \circ \psi : \widehat{D} \to \widehat{A}$. Let
$\Gamma$ be the twist associated to $(C,D)$ as in
Theorem~\ref{thm-alex}. Then we define
\[
\delta(A) := (\pi^*_{\tilde\psi})^{-1}\big(\partial^1\big(\!\big[\underline{\Gamma}\big]\!\big)\big) \in H^2(\widehat{A}, \Tgerms).
\]
\end{defn}

\begin{remark}\label{rmk:why not DD invariant}
It seems difficult to establish that our invariant $\delta(A)$
coincides with the original Dixmier-Douady invariant of $A$
when $A$ is a continuous-trace $C^*$-algebra. The issue is that
the boundary map $\partial^1$ which takes the class of a twist
over $R(\tilde\psi)$ to an element of
$H^2(R(\tilde\psi),\Tgerms)$ is defined by abstract nonsense.
Nevertheless our invariant does classify Fell algebras up to
spectrum-preserving Morita equivalence (see Theorem~\ref{thm:DD
classification}), and this generalises the original
Dixmier-Douady theorem of~\cite{DD}.
\end{remark}

\begin{prop}
Let $(G,X)$ be a free Cartan transformation group. Then
$\delta(C_0(X) \rtimes G) = 0$ as an element of $H^2(X/G,
\Tgerms)$.
\end{prop}

\begin{proof}
By Corollary~\ref{cor-greens}, $C_0(X)\rtimes G$ is Morita
equivalent to a  groupoid $C^*$-algebra $C^*(R)$, where $R$ is
a principal, \'etale groupoid. By the remarks following
Corollary~\ref{cor-greens}, $C^*(R)$ is a Fell algebra. Thus
the reduced $C^*$-algebra $C^*_{\red}(R)$ is also a Fell
algebra and hence is nuclear. By \cite[Corollary~6.2.14]{ADR},
since $R$ principal and $C^*_{\red}(R)$ is nuclear, $R$ is
measurewise amenable, and thus $C^*(R)=C^*_{\red}(R)$ by
\cite[Proposition~6.1.8]{ADR}.

By Lemma~\ref{lem:C* trivial twist}, $C^*_{\red}(R)$ is
isomorphic to the $C^*$-algebra  $\tgcsa{R \times \TT}{R}$ of the trivial twist $\Gamma:=R\times\TT\to R$. The associated sheaf
twist $\underline{\Gamma}$ is therefore also trivial and hence
$\partial^1([\underline{\Gamma}])=0$. It follows that
$\delta(C_0(X) \rtimes G) = 0$ also.
\end{proof}

To prove our classification theorem, we need another lemma.

\begin{lemma}\label{lem:DD->equivalent twists}
Let $X$ be a second-countable, locally compact, locally
Hausdorff space. For $i =1,2$, let $Y_i$ be a second-countable,
locally compact, Hausdorff space, and let $\psi_i : Y_i \to X$
be a local homeomorphism. For $i = 1,2$, let $\Gamma_i \to
R(\psi_i)$ be a twist, and suppose that the isomorphism
$\iota_{1,2}$ of Corollary~\ref{cor:Br well-defd} carries
$\partial^1\big(\!\big[\underline{\Gamma}_1\big]\!\big)$ to
$\partial^1\big(\!\big[\underline{\Gamma}_2\big]\!\big)$. Then
there exists a locally compact, Hausdorff space $Z$ and local
homeomorphisms $\rho_i : Z \to Y_i$ such that $\psi_1 \circ
\rho_1 = \psi_2 \circ \rho_2$ and $\pi_{\rho_1}^*(\Gamma_1)
\cong \pi_{\rho_2}^*(\Gamma_2)$ as twists over $R(\psi_1 \circ
\rho_1)$. In particular, $\Gamma_1$ and $\Gamma_2$ are
equivalent twists.
\end{lemma}
\begin{proof}
Let $Y := Y_1 \ast Y_2 = \{(y_1, y_2) \in Y_1 \times Y_2 :
\psi_1(y_1) = \psi_2(y_2)\}$. For each $i$, let $\phi_i : Y \to
Y_i$ be the projection map; then $\psi_1 \circ \phi_1 = \psi_2
\circ\phi_2$ is a local homeomorphism from $Y$ to $X$.

We claim that each $\Gamma_i$ is twist-equivalent to
$\pi^*_{\phi_i}(\Gamma_i)$. To see this, we first observe that
for $i = 1,2$, the assignment
\[
(x, (\phi_i(x), \phi_i(y)), y) \mapsto (x,y)
\]
is an isomorphism from $R(\psi_i)^{\phi_i}$ to $R(\psi_i \circ
\phi_i)$, and the assignment $((x,y),g) \mapsto (x,g,y)$ is an
isomorphism from $\pi_{\phi_i}^*(\Gamma_i)$ to
$\Gamma_i^{\phi_i}$. By \cite[Proposition~5.7]{Kumjian1986},
each $\Gamma_i$ is equivalent to $\Gamma_i^{\phi_i}$, so the
isomorphisms above complete the proof of the claim.

Since
\[
\iota_{1,2}\big(\partial^1\big(\!\big[\pi_{\phi_1}^*(\underline{\Gamma}_1)\big]\!\big)\big)
    = \partial^1\big(\!\big[\pi_{\phi_2}^*(\underline{\Gamma}_2)\big]\!\big),
\]
\cite[Proposition~3.9]{Kumjian1988} implies that there exists a
locally compact, Hausdorff space $Z$ and a local homeomorphism
$\tau : Z \to Y$ such that
$\pi_\tau^*(\pi_{\phi_1}^*(\underline{\Gamma}_1))$ and
$\pi_\tau^*(\pi_{\phi_2}^*(\underline{\Gamma}_2))$ are
isomorphic sheaf twists. Since each $\pi^*_{\phi_i \circ
\tau}(\underline{\Gamma}_i) =
\pi_\tau^*(\pi_{\phi_i}^*(\underline{\Gamma}_i))$, it follows
that with $\rho_i := \phi_i \circ \tau :  Z \to Y_i$, we have
$\pi^*_{\rho_1}(\underline{\Gamma}_1) \cong
\pi^*_{\rho_2}(\underline{\Gamma}_2)$, and hence by naturality
\begin{equation}\label{eq:pisubrhos equiv}
\pi^*_{\rho_1}(\Gamma_1) \cong \pi^*_{\rho_2}(\Gamma_2).
\end{equation}

For the final assertion, we apply the claim above with $\phi_i$
replaced with $\rho_i$ to see that each $\Gamma_i$ is
twist-equivalent to $\pi^*_{\rho_i}(\Gamma_i)$, and then
invoke~\eqref{eq:pisubrhos equiv}.
\end{proof}

\begin{thm}\label{thm:DD classification}
Let $A_1$ and $A_2$ be separable Fell algebras. Then $A_1$ and
$A_2$ are Morita equivalent if and only if there is a
homeomorphism $h : \widehat{A}_1 \to \widehat{A}_2$ such that
the induced isomorphism $h^* : H^2(\widehat{A}_2, \Tgerms) \to
H^2(\widehat{A}_1, \Tgerms)$ carries $\delta(A_2)$ to
$\delta(A_1)$.
\end{thm}
\begin{proof}
First suppose that $H$ is an $A_2$--$A_1$-imprimitivity bimodule
and let $h : \widehat{A}_1 \to \widehat{A}_2$ be the associated
Rieffel homeomorphism. Let $(C, D)$ be a diagonal pair together
with an $A_2$--$C$-imprimitivity bimodule $K$, and let $k :
\widehat{C} \to \widehat{A}_2$ be the Rieffel homeomorphism
associated to $K$. Let $\psi : \widehat{D} \to \widehat{C}$ be
the spectral map.

Let $\tilde\psi_2 := k \circ \psi : \widehat{D} \to
\widehat{A}_2$, and let $\Gamma_2\to R$ be the twist obtained
from $(C, D)$ as in Theorem~\ref{thm-alex}.  Note that
$R=R(\tilde\psi_2)$ by
Proposition~\ref{thm-equivalencerelationoftwist}. By
definition, $\delta(A_2) = (\pi_{\tilde{\psi}_2}^*)^{-1}\big(
\partial^1\big(\!\big[\underline{\Gamma}_2\big]\!\big)\big) \in
H^2(\widehat{A}_2, \Tgerms)$. Let $\tilde{H}$ be the dual
bimodule of $H$, and observe that $K \otimes_{A_2} \tilde H$ is
a $C$--$A_1$-imprimitivity bimodule with Rieffel homeomorphism
$h^{-1} \circ k$. Let $\tilde\psi_1 := h^{-1} \circ k \circ \psi
: \widehat{D} \to \widehat{A}_1$, and let $\Gamma_1$ be the
twist over $R(\tilde\psi_1)$ obtained from $(C, D)$ as in
Theorem~\ref{thm-alex}. Again by definition, $\delta(A_1) =
(\pi_{\tilde\psi_1}^*)^{-1}\big(\partial^1\big(\!\big[\underline{\Gamma}_1\big]\!\big)\big)
\in H^2(R(\tilde\psi_1), \Tgerms)$. Since $\tilde\psi_2 = h
\circ \tilde\psi_1$, Theorem~\ref{thm:DD class well-defd} and
Proposition~\ref{prp:Br(X)=H2(R(psi))} imply that the induced
isomorphism $h^* : H^2(\widehat{A}_2, \Tgerms) \to
H^2(\widehat{A}_1,\Tgerms)$ carries $\delta(A_2)$ to
$\delta(A_1)$.

Now suppose that there is a homeomorphism $h : \widehat{A}_1 \to
\widehat{A}_2$ such that the induced isomorphism $h^* :
H^2(\widehat{A}_2,\Tgerms) \to H^2(\widehat{A}_1,\Tgerms)$
carries $\delta(A_2)$ to $\delta(A_1)$. Let $(C_i, D_i)$ be
diagonal pairs with $C_i$ Morita equivalent to $A_i$, let
$\psi_i : \widehat{D}_i \to \widehat{C}_i$ be the spectral maps,
and let $\Gamma_i \to R(\psi_i)$ be the associated twists.
Proposition~\ref{prp:Br(X)=H2(R(psi))} and the hypothesis that
$h^*$ carries $\delta(A_2)$ to $\delta(A_1)$ ensures that the
induced map (also denoted $h^*$) from $H^2(R(\psi_2), \Tgerms)$
to $H^2(R(\psi_1),\Tgerms)$ satisfies
\[
h^*\big(\partial^1\big(\!\big[\underline{\Gamma}_2\big]\!\big)\big)
    = \partial^1\big(\!\big[\underline{\Gamma}_1\big]\!\big).
\]
Hence we may regard $\Gamma_1$ as a twist over $R(h \circ
\psi_1)$ with the same image under $\partial^1$ as $\Gamma_2$.
Lemma~\ref{lem:DD->equivalent twists} therefore implies that
$\Gamma_1$ and $\Gamma_2$ are equivalent twists, and then
Lemma~\ref{lem-twists} implies that $C_1$ and $C_2$ are Morita
equivalent, whence $A_1$ and $A_2$ are also Morita equivalent.
\end{proof}

Recall that if $X$ is the spectrum of a $C^*$-algebra then $X$
is locally compact, and every open subset of $X$ is itself
locally compact (because it is the spectrum of an ideal); such
spaces are called \emph{locally quasi-compact} in
\cite[\S3.3]{dix}.

\begin{remark}\label{rmk:cohomologies coincide}
Let $X$ be a second-countable, locally Hausdorff space such
that every open subset of $X$ is locally compact. We will show
that every element of $H^2(X,\Tgerms)$ arises as the class of a
Fell algebra with spectrum $X$. To do this, it is convenient to
work with \v{C}ech cohomology, rather than sheaf cohomology, of
a locally compact, Hausdorff ``desingularisation" $Y$ of $X$.

It is observed in \cite[Hooptedoodle~4.16]{tfb}, with reference
to \cite[\S5.23]{Warner1971}, that all reasonable
sheaf-cohomology theories coincide over Hausdorff paracompact
spaces. Specifically, by Theorem~5.32 of \cite{Warner1971} and
the subsequent corollary, any two sheaf cohomologies over
Hausdorff paracompact spaces satisfying
\cite[Axioms~5.18]{Warner1971} are canonically isomorphic.
Warner demonstrates in \cite[\S5.33]{Warner1971} that \v{C}ech
cohomology satisfies these axioms. All but one of these axioms
are automatically satisfied by the sheaf cohomology used here
because it is defined in terms of derived functors; the
remaining axiom (property~(b) of
\cite[Axioms~5.18]{Warner1971}) requires that $H^q(Y,B) = 0$
for $q > 0$ if $B$ is a fine sheaf, and this follows from
\cite[Proposition~4.36]{Voisin2002}.
\end{remark}

For an introduction to \v{C}ech cohomology, see
\cite[Chapter~4]{tfb}. Given a covering $\{U_i : i \in I\}$ of a
space $Y$, and given $i,j,k \in I$, we write $U_{ijk}$ for the
intersection $U_i \cap U_j \cap U_k$.

\begin{lemma}\label{lem:pullback to trivialise}
Let $Y$ be a second-countable, locally compact, Hausdorff space.
For each $a \in H^2(Y, \Tgerms)$, there exists a locally
compact, Hausdorff space $Z$ and a local homeomorphism $\phi : Z
\to Y$ such that $\phi^*(a) = 0 \in H^2(Z, \Tgerms)$.
\end{lemma}
\begin{proof}
By Remark~\ref{rmk:cohomologies coincide}, we may regard $a$ as
an element of $\check{H}^2(Y, \Tgerms)$. So there exists a
covering $\mathcal{U} = \{U_i : i \in I\}$ of $Y$ by open sets
and a 2-cocycle $c = \{c_{ijk} : U_{ijk} \to \TT \mid i,j,k \in
I\}$ such that $a$ is equal to the class of $c$ in
$\check{H}^2(Y, \Tgerms)$.

Let $Z := \bigcup_{i \in I} (\{i\} \times U_i) \subset I \times
Y$, and let $\phi : Z \to Y$ be the projection onto the second
coordinate. Let $V_i := \{i\} \times U_i \subset Z$ for each
$i$. Then $\mathcal{V} = \{V_i : i \in I\}$ is a refinement of
the pullback cover $\{\phi^{-1}(U_i) : i \in I\}$ to a cover by
mutually disjoint sets; in particular the only nonempty triple
overlaps are those of the form $V_{iii}$. Since $\check{H}^2(Z,
\Tgerms)$ is the direct limit over covers of $Z$ of the cocycle
group, the class of $\phi^*(c)$ is equal to the class of its
image $i_{\mathcal{U}, \mathcal{V}}(\phi^*(c))$ in the cocycle
group for the $\mathcal{V}$. Since the $V_i$ are pairwise
disjoint, $i_{\mathcal{U}, \mathcal{V}}(\phi^*(c))$ amounts to
a continuous circle-valued function on each $V_{iii}$, and so
is a coboundary (specifically, the coboundary of itself
regarded as a $1$-cochain).
\end{proof}

Next we  require notation for the
forgetful functor which takes an equivariant $\Gamma$-sheaf
$B$ to an ordinary sheaf $B^0$ over $\Gamma^0$ by forgetting
the $\Gamma$-action. Note that  $B^0 = \jmath^*(B)$ where
$\jmath : \Gamma^{(0)} \to \Gamma$ is the inclusion map.
The pullback functor induces the homomorphism
$\jmath^*_n: H^n(\Gamma, B)\to H^n(\Gamma^{(0)}, B^0)$ which
appears in the long exact sequence of
\cite[Theorem~3.7]{Kumjian1988}.

\begin{prop}\label{prop:invariant exhausted}
Let $X$ be a second-countable, locally Hausdorff space such
that every open subset of $X$ is locally compact. Then for each
$a \in H^2(X,\Tgerms)$ there exists a locally compact,
Hausdorff space $Z$, a local homeomorphism $\psi : Z \to X$ and
a twist $\Gamma$ over $R(\psi)$ such that $a =
(\pi_\psi^*)^{-1}\big(\partial^1\big(\!\big[\underline{\Gamma}\big]\!\big)\big)$.
In particular, for each $a \in H^2(X,\Tgerms)$, there exists a
separable Fell algebra $A$  such that $\widehat{A} = X$ and $a = \delta(A)$.
\end{prop}

\begin{proof}
Choose a countable open cover $\{U_i\}$ of $X$ consisting of
Hausdorff subsets of $X$ and let $Y:=\bigsqcup_i U_i$. Since
every open subset of $X$ is locally compact, each $U_i$ is
locally compact, and hence $Y$ is locally compact and
Hausdorff. The inclusion map $\theta:Y\to X$ is a local
homeomorphism. Let $b := \pi_\theta^*(a) \in
H^2(R(\theta),\Tgerms)$. By Lemma~\ref{lem:pullback to
trivialise}, there exists a second-countable, locally compact,
Hausdorff space $Z$ and a local homeomorphism $\phi : Z \to Y$
such that $\phi^*(\jmath^*_2(b)) = 0$. Let $\psi := \theta
\circ \phi : Z \to X$. Then by naturality of the long exact
sequence, $\jmath^*_2(\pi_\psi^*(a)) = 0 \in H^2(Z,\Tgerms)$.
By exactness, it follows that  there is a twist $\Gamma$ over
$R(\psi)$ such that $\pi_\psi^*(a)=
\partial^1\big(\!\big[\underline{\Gamma}\big]\!\big)$.
Let $A := \tgcsa{\Gamma}{R(\psi)}$. By Theorem~\ref{thm-main}(2), $A$ is
a Fell algebra and its spectrum is homeomorphic to $X$.
After identifying $\hat A$ with $X$, $\delta(A) = a$ by Definition \ref{dfn:DD invariant}.
\end{proof}

\begin{remark}\label{rmk:kudos to KMRW}
There is a notion of a Brauer group $\Br(G)$ for a locally
compact, Hausdorff groupoid $G$ \cite{KMRW}. Moreover,
\cite[Proposition~11.3]{KMRW} implies that if $G$ is \'etale,
then $\Br(G) \cong H^2(G, \Tgerms)$. If $Z$ is a groupoid
equivalence of locally compact, Hausdorff groupoids $G$ and
$H$, then $Z$ determines an isomorphism between $H^2(G,
\Tgerms)$ and $H^2(H, \Tgerms)$ \cite[Theorem~4.1]{KMRW}. Thus
$H^2(X, \Tgerms)$ is canonically isomorphic to $\Br(R(\psi))$ for any local
homeomorphism $\psi$ from a locally compact, Hausdorff space onto
$X$ (the isomorphism of Proposition~\ref{prp:H2 well-defd}
is a special case of \cite[Theorem~4.1]{KMRW}).

Though it would have been natural to identify $\Br(X)$ with
$H^2(X, \Tgerms)$ for a  locally compact, locally Hausdorff
space $X$, we have chosen not to use the notation $\Br(X)$ nor
the term Brauer group as the notion has not yet been extended
to non-Hausdorff spaces (to say nothing of non-Hausdorff
groupoids).   To justify the use of the term it  would first be
necessary to formulate a notion of balanced tensor product for
Fell algebras with spectra identified with $X$. We leave the
details for future work.
\end{remark}

\appendix\section{The \texorpdfstring{$C^*$}{C*}-algebra of a twist}

Let $\Gamma$ be a $\TT$-groupoid and $ \Gamma \stackrel{q}{\to}
R $ a twist (see page~\pageref{defn twist} for the definition
of a twist).  The details of the construction of the twisted
groupoid $C^*$-algebra $\tgcsa{\Gamma}{R}$ may be found in \S2
of \cite{Kumjian1986}. The idea is that a dense subalgebra is
identified with continuous compactly supported sections of an
associated line bundle, and then convolution and involution are
defined by virtue of the Fell bundle structure of the line
bundle (as in \cite[\S{5}]{Renault2008} but our conventions
differ slightly). We briefly review the construction for the
convenience of the reader. Since $R$ is an \'etale groupoid, we
may use the standard Haar system consisting of counting
measures.

Define a line bundle $L = L(\Gamma)$ over $R$ by taking the
quotient of $\CC \times \Gamma$ by the diagonal action of $\TT$
--- that is, $L$ consists of equivalence classes of the
equivalence relation $(z,\gamma)\sim(tz, t\cdot\gamma)$ for
$t\in\TT$. Then $L$ is a complex line-bundle over $R$ with
bundle map $[(z,\gamma)] \mapsto q(\gamma)$. As usual, we denote
by $L_\rho$ the fibre over $\rho \in R$.

The following Fell bundle structure on $L$ is implicit in
\cite{Kumjian1986}. Given a composable pair $(\rho_1, \rho_2)$
of elements in $R$ and elements $[(z_i,\gamma_i)] \in
L_{\rho_i}$, we define the product in $L_{\rho_1\rho_2}$ by
\[
[(z_1,\gamma_1][(z_2,\gamma_2)] = [(z_1z_2,\gamma_1\gamma_2)];
\]
and involution is defined by $[(z,\gamma)] \in L_{q(\gamma)}
\mapsto [(\overline{z},\gamma^{-1})] \in L_{q(\gamma^{-1})}$.
It is straightforward to check that these operations are
well defined.

Now define
\[
C_c(\Gamma; R):=\{f\in C_c(\Gamma): f(t\cdot\gamma)
    = tf(\gamma)\text{\ for $t\in\TT$ and $\gamma\in \Gamma$}\}.
\]
Each $f \in C_c(\Gamma;R)$ determines a section $\tilde{f}$ of
$L$ by the formula $\tilde{f}(q(\gamma)) :=
[(f(\gamma),\gamma)]$ (it is straightforward to check that this
map is well defined). Moreover, given $\gamma \in \Gamma$, and
$z \in \TT$ the element $z$ is uniquely determined by $\gamma$
and $[z,\gamma]$, so $f \mapsto \tilde{f}$ is a bijection.
Hence we may endow $C_c(\Gamma; R)$ with the structure of a
$*$-algebra by the following formulae for $f, g \in C_c(\Gamma;
R)$
\[
(f*g)\widetilde{\ }\;(\rho)
    = \sum_{\alpha\beta = \rho} \tilde{f}(\alpha)\tilde{g}(\beta)
\quad\text{and}\quad
    \widetilde{f}^*(\rho) = \tilde{f}(\rho^{-1})^*.
\]
These operations match up with the convolution and involution
on $C_c(\Gamma;R)$ used in, for example, \cite{MW1992} and
\cite{Renault2008}. To keep our notation simple we identify
each element of $C_c(\Gamma;R)$ with the corresponding
compactly supported continuous section of the Fell bundle $L$
(as in \cite{Kumjian1986} and \cite{Renault2008}).

Note that the map, $(x, z) \mapsto [(x, z)]$ gives a
trivialisation $R^{(0)} \times \CC \cong L|_{R^{(0)}}$ and hence
$L$ is trivial over $R^{(0)}$; thus we may identify
$C_c(R^{(0)})$ with the abelian subalgebra $\{f\in C_c(\Gamma;
R):\supp f\subset R^{(0)}\}$, so  $C_c(\Gamma; R)$ may be
regarded as a right $C_c(R^{(0)})$ module under
right-multiplication. Moreover, the restriction map
$P:C_c(\Gamma; R)\to C_c(R^{(0)})$  is a $C_c(R^{(0)})$-module
morphism. For $f,g\in C_c(\Gamma; R)$, the formula $\langle
f\,,\, g\rangle =P(f^*g)$ defines an inner product on
$C_c(\Gamma; R)$, and the completion  $H(\Gamma; R)$ of
$C_c(\Gamma; R)$ in the norm $\|f\|=||\langle f\,,\,
f\rangle\|_\infty^{1/2}$  is a right-Hilbert
$C_0(R^{(0)})$-module.  Finally,  left multiplication by $f\in
C_c(\Gamma; R)$ extends to an adjointable operator $\phi(f)$ on
$H(\Gamma; R)$; this defines a $*$-homomorphism $\phi :
C_c(\Gamma; R) \to \Ll (H(\Gamma; R))$. The twisted groupoid
$C^*$-algebra \tgcsa{\Gamma}{R} is defined to be the completion
of $C_c(\Gamma; R)$ in the operator norm, and $C_0(R^{(0)})$ is
identified with the closure of $C_c(R^{(0)})$ in
$\tgcsa{\Gamma}{R}$.

We show that $(\tgcsa{\Gamma}{R}, C_0(R^{(0)}))$ is a diagonal
pair in the sense of Definition~\ref{dfn:diagonal}. This follows
from \cite[Proposition~2.9]{Kumjian1986} and
Corollary~\ref{cor:diagonals are diagonals} once we establish
that $C_0(R^{(0)})$ contains an approximate identity for
$\tgcsa{\Gamma}{R}$. For this, let $(K_n)^\infty_{n=1}$ be an
increasing sequence of compact subsets of $R^{(0)}$ such that
$R^{(0)}=\bigcup_n K_n$, and for each $n \in \NN$, fix $g_n\in
C_c(R^{(0)})$ such that $g_n|_{K_n} = 1$.  Since
$(fg_n)(\rho)=f(\rho)g_n(s(\rho))$ for each compactly supported
section $f$, the $g_n$ form an approximate identity for
\tgcsa{\Gamma}{R}.

For the next result, recall from \cite[Remark~4.2]{Kumjian1988}
that a trivial twist over $R$ is isomorphic to $R \times \TT \to
R$.   The reduced norm on $C_c(R)$ is variously defined in the
literature; see, for example, \cite[p.~82]{Renault1980} and
\cite[p.~146]{ADR}, and also \cite[\S{3}]{sims-williams} for a
discussion of the equivalence of these two definitions. There
are also two definitions of the reduced norm on $C_c(\Gamma;R)$:
one using the operator norm outlined above and the other based
on  induced representations from point evaluations on
$C_0(R^{(0)})$  used in \cite[p.~40]{Renault2008}. The
equivalence of the two follows from the observation  that
\[
    (\pi_x(f)\xi|_{s^{-1}(x)} \,|\, \eta|_{s^{-1}(x)})_{H_x}
        = \overline{\langle \phi(f)\xi\,,\, \eta\rangle_{C_0(R^{(0)})}(x)}
\]
for compactly supported sections $f\in C_c(\Gamma;R)$ and
$\xi,\eta\in C_c(\Gamma;R)\subset H(\Gamma;R)$ (see also the
discussion in  \cite[p.~40]{Renault2008}).

\begin{lemma}\label{lem:C* trivial twist}
If $\Gamma \to R$ is a trivial twist, then $\tgcsa{\Gamma}{R}
\cong C^*_{\red}(R)$.
\end{lemma}
\begin{proof}
Suppose that $\Gamma$ is a trivial twist.  Then we may identify
$L$ and $\CC \times R$ and therefore $C_c(R,L)$ and $C_c(R)$.
It is routine to check that this identification preserves the
$*$-algebra structure defined above. So we just need to check
that for $f \in C_c(R)$ we have $\| f \|_r = \| \phi(f) \|$.
Let $f \in C_c(R)$. By the definition given in
\cite[p.~82]{Renault1980}, we have
\[\textstyle
\| f \|_r = \sup_\mu \| \Ind_\mu (f) \|
\]
where $\mu$ ranges over Radon measures on $R^{(0)}$. Denote by $\pi_\mu : C_0(R^{(0)}) \to
\Bb(L^2(R^{(0)}, \mu))$ the usual representation by multiplication operators. The discussion on
page 81 of \cite{Renault1980}  shows that the induced representation $\Ind_\mu$ is given on
\[
H(\Gamma;R) \otimes_{\pi_\mu} L^2(R^{(0)}, \mu)
\]
by the formula
\[
\Ind_\mu(f)(\xi \otimes g) = \phi(f)\xi\otimes g.
\]
Hence, $ \| \Ind_\mu (f) \| \le \| \phi(f) \|$ 
and so $\| f \|_r \le \| \phi(f) \|$. Now, let $\mu$ be a
measure with full  support; then $\pi_\mu$ is faithful and
hence the corresponding representation of $\Kk(H(\Gamma;R))$ is
also faithful.  Since $\Ll(H(\Gamma; R)) =
M(\Kk(H(\Gamma;R)))$, this shows that $\Ind_\mu$ is faithful on
$\tgcsa{\Gamma}{R}$.  Hence, $\| \phi(f)\| = \| \Ind_\mu(f)\|
\le \| f \|_r$.
\end{proof}

\end{document}